\pdfoutput=1
\documentclass[12pt]{article}
\usepackage{lineno,hyperref}
\usepackage{mathtools}
\usepackage{amssymb}
\usepackage{eqnarray,amsthm}
\usepackage{enumerate}
\usepackage{arydshln}
\usepackage{chngpage}
\usepackage{tikz} 
\usepackage{tikz-cd}
\usepackage{fancyhdr}
\usepackage[short, nodayofweek]{datetime}
\makeatletter
\@addtoreset{chapter}{part}
\makeatother
\newtheorem{manualtheoreminner}{Theorem}
\newenvironment{manualtheorem}[1]{%
	\renewcommand\themanualtheoreminner{#1}%
	\manualtheoreminner
}{\endmanualtheoreminner}
\newtheorem{theorem}{Theorem}[section]
\newtheorem{lemma}[theorem]{Lemma} 
\newtheorem{proposition}[theorem]{Proposition} 
\newtheorem{corollary}[theorem]{Corollary} 
\theoremstyle{definition}
\newtheorem{definition}[theorem]{Definition}
\theoremstyle{remark}
\newtheorem{remark}[theorem]{Remark}

\DeclareMathOperator{\Ima}{Im}

\DeclareMathOperator{\Ker}{Ker}
\title{
	Continuous Cocycles Endowed with Point-Open Topology\thanks{The author obtained the results in this paper while studying for a PhD degree at the University of Exeter funded by an EPSRC Doctoral Training Grant. This research was partially supported by the ERC Advanced grant 320974.} 
}
\author{Kayvan Nejabati Zenouz\thanks{School of Mathematics, University of Edinburgh, 5605, JCMB, Edinburgh, EH9 3FD, Email: knejabat@ed.ac.uk, Website: \href{http://www.maths.ed.ac.uk/school-of-mathematics/people?person=644}{http://www.maths.ed.ac.uk/school-of-mathematics/people?person=644}}}
\date{\today}
\begin{document}
	\maketitle
	\begin{abstract}
	Given a topological group $ G $ and a Hausdorff topological group $ A $ on which $ G $ acts continuously and compatibly with the group operation of $ A $, we study the set of continuous cocycles of $ G $ with value in $ A $. This set is a function space and can be endowed with several topologies. By imposing a suitable function space topology on the set of cocycles of $ G $ with value in $ A $, we propose a topological study of this set, and we prove, as our first main result, that if $ A $ is a compact group having a presentation as an inverse limit of compact and Hausdorff topological groups $ A_{r} $, for $ r $ in a directed poset $ R $, on which $ G $ acts continuously and compatibly with the group operation of $ A_{r} $ and equivariantly with respect to the transition maps, then one has a natural identification between the first nonabelian cohomology set of $ G $ with coefficients in the inverse limit $ A $ and the inverse limit of first nonabelian cohomology sets of $ G $ with coefficients in $ A_{r} $. Furthermore, we prove, as our second main result, that if $ G $ is compact and Hausdorff, and $ A $ is abelian - therefore one can define cohomology groups for all $ n\geq1 $ - then under a certain condition on $ A_{r} $ one has a natural identification between cohomology group with coefficients in the inverse limit $ A $ and the inverse limit of cohomology groups with coefficients in $ A_{r} $, for all $ n\geq1 $.
	\end{abstract}
	\tableofcontents{}
	\section{Introduction}\label{S1}
	It is a theorem cf. \cite[p.~141, (2.7.5) Theorem]{MR2392026} that if $ G $ is a profinite group, and $ A $ is a compact $ G $-module having a presentation $ A=\varprojlim_{n \in \mathbb{N}}A_{n} $ as a countable inverse limit of finite discrete $ G $-modules $ A_{n} $, then there exist a natural exact sequence  \[ 
	\begin{tikzcd}[row sep=1.2em , column sep=1.2em]
	1 \arrow{r}[]{}\arrow{r}{}[swap]{} & \varprojlim_{n \in \mathbb{N}}^{1}H^{i-1}(G,A_{n})   \arrow{r}[]{} & H^{i}_{\text{cts}}(G,A) \arrow{r}[]{} & \varprojlim_{n \in \mathbb{N}}H^{i}(G,A_{n}) \arrow{r}[]{} & 1; 
	\end{tikzcd} \] 
	however, in the investigations presented in \cite[p.~136, \S7]{MR2392026} it has not been taken into consideration the fact that there are several natural ways to make $ H^{i}_{\text{cts}}(G,A) $ into a topological space, and to study some of the topological properties of $ H^{i}_{\text{cts}}(G,A) $; hence the fundamental aim of this paper is to investigate generalisations of the above theorem by introducing a suitable topology on the groups $ H^{i}_{\text{cts}}(G,A) $, and studying their topological properties. However, we first begin by studying cohomology with coefficients in nonabelian groups, and then move onto cohomology with coefficients in abelian groups, and eventually, as a corollary of our results, we give a relationship between $ G^{i-1} $ and $ A_{n} $ for all $ n $ under which one has $ \varprojlim_{n \in \mathbb{N}}^{1}H^{i-1}(G,A_{n})=1 $ on the exact sequence above. 
	
	We start with a continuous action of a topological group $ G $ on a Hausdorff topological group $ A $ compatible with the group operation of $ A $ (in such case we shall refer to $ A $ as a $ G $-group), and we study the properties of the topological space $ Z^{1}_{\text{cts}}(G,A) $, the set of continuous 1-cocycles of $ G $ in $ A $, endowed with the point-open topology (aka the topology of pointwise convergence); we also study the quotient space $ H^{1}_{\text{cts}}(G,A) $, of cohomology classes of continuous cocycles of $ G $ in $ A $, endowed with the quotient topology. In particular, investigating the properties of the quotient space $ H^{1}_{\text{cts}}(G,A) $ we prove, as our first main result the following.
	\begin{manualtheorem}{1}
		Let $ A $ be a compact $ G $-group. Assume $ A $ has a presentation $ A= \varprojlim_{\stackrel{ \ }{r\in R}}A_{r} $. Denote by $ \varphi_{r}: A\longrightarrow A_{r} $ the natural projections. Then there exists a natural continuous bijection \[ \Theta : H_{\text{cts}}^{1}(G,A)_{\text{po}} \longrightarrow \varprojlim_{\stackrel{ \ }{r\in R}} H_{\text{cts}}^{1}(G,A_{r})_{\text{po}} \ \ \text{given by} \ \ [a]\longmapsto ([\varphi_{r\ast}(a)]). \]
		In particular, if $ G $ is a compact (even locally compact) and Hausdorff space, and $ A $ is evenly continuous with respect to $ G $, then $ \Theta $ is a homeomorphism.
	\end{manualtheorem}

    That is if $ A=\varprojlim_{\stackrel{ \ }{r \in R}}A_{r} $ is given as an inverse limit of compact topological $ G $-groups $ A_{r} $, for $ r $ in a directed poset $ R $, then there exists a natural continuous bijection of pointed sets \[ \Theta=\Theta_{1} : H^{1}_{\text{cts}}(G,A)_{\text{po}}\longrightarrow \varprojlim_{\stackrel{ \ }{r \in R}}H^{1}_{\text{cts}}(G,A_{r})_{\text{po}},\] 
	which is a homeomorphism, when $ G $ a is compact (even locally compact) and Hausdorff group and the set of continuous maps from $ G  \longrightarrow A_{r} $ is an evenly continuous set for each $ r \in R $.
	
	Next we apply our methodology to obtain some results in abelian cohomology. In the case where $ A $ is a compact abelian $ G $-group, i.e., $ A $ is a compact $ G $-module, for a fixed $ n\geq 1 $, we study the group of $ n $-cocycles of $ G $ in $ A $ endowed with the point-open topology, and we prove if $ A=\varprojlim_{\stackrel{ \ }{r \in R}}A_{r} $ is given as above, $ G $ is compact and Hausdorff, and when the set of continuous function from $ G^{n-1}\longrightarrow A_{r} $ is evenly continuous for all $ r \in R $, then that there exists a continuous group isomorphism \[ \Theta_{n} : H^{n}_{\text{cts}}(G,A)_{\text{po}}\longrightarrow \varprojlim_{\stackrel{ \ }{r \in R}}H^{n}_{\text{cts}}(G,A_{r})_{\text{po}},\] 
	which is a homeomorphism, when the set of continuous function from $ G^{n}\longrightarrow A_{r} $ is evenly continuous for all $ r \in R $. More precesely we prove as our second main theorem the following.
	\begin{manualtheorem}{2}
		Let $ A $ a compact $ G $-module and fix $ n\geq 1 $. Assume $ A $ has a presentation $ A= \varprojlim_{\stackrel{ \ }{r\in R}}A_{r} $, where $ A_{r} $ is evenly continuous with respect to $ G^{n-1} $ for all $ r\in R $. Denote by $ \varphi_{r}: A\longrightarrow A_{r} $ the natural projections. Then there exists a natural continuous bijection \[ \Theta_{n} : H_{\text{cts}}^{n}(G,A)_{\text{po}} \longrightarrow \varprojlim_{\stackrel{ \ }{r\in R}} H_{\text{cts}}^{n}(G,A_{r})_{\text{po}} \ \ \text{given by} \ \ [a]\longmapsto ([\varphi_{r\ast}(a)]). \]
		In particular, if $ A $ is evenly continuous with respect to $ G^{n} $, then $ \Theta_{n} $ is a homeomoprhism.
	\end{manualtheorem} 
	
	This document is organised as follows. In Section \ref{S2} we provide background information relating to function spaces, and we develop a general theory for studying the topological space of continuous 1-cocycles and the first nonabelian continuous cohomology space endowed with point-open topology. We prove several topological properties of these spaces; we also briefly discuss the, well-known, one-to-one correspondence between the set of continuous 1-cocycles and the set of continuous torsors. In Section \ref{S3} we focus on torsors under a compact and Hausdorff topological group endowed with the point-open topology and provide the proof for our first main theorem, we also we discuss two immediate applications of our first main theorem. Finally, in Section \ref{S4} we investigate the abelian situation, we prove our second main theorem, and provide an application of this theorem.  
	\section{Continuous Cocycles and Torsors}\label{S2}
	\subsection{Function Spaces}\label{SB1}
	Let $ A $ and $ B $ be topological spaces. Denote by $ M(B,A) $ the set of all maps from $ B \longrightarrow A $. As a set we have $ M(B,A)=\prod_{b \in B}A_{b} $ i.e., $ M(B,A) $ is the product of copies of $ A $ one for each $ b \in B $. It is known that the set $ M(B,A) $ can be endowed with several topologies. For example $ M(B,A) $ can be endowed with the compact-open topology, the point-open topology, and if $ A $ is a uniform space; e.g., a metric space or a topological group, then the set $ M(B,A) $ can also be endowed the topology of uniform convergence, or the topology of compact convergence. 
	
	A subbase for compact-open topology on $ M(B,A) $ is given by subsets of the form $ \beta(K,V)\stackrel{\mathrm{def}}{=}\{f\in M(B,A): f(K) \subseteq V \} $, where $ K \subseteq B $ is a compact subspace and $ V \subseteq A $ is an open subset. Therefore, an element of the base for the compact-open topology on $ M(B,A) $ is given by a finite intersection of elements of the subbase cf. \cite[p.~221]{MR0370454}. We shall write $ M(B,A)_{\text{co}} $ for the set $ M(B,A) $ endowed with compact-open topology. 
	
	A coarser/smaller topology than compact-open on $ M(B,A) $ is the point-open topology. The point-open topology on $ M(B,A) $ is obtained by taking as subbase subsets of the form $ \beta(\{b\},V) $, for a point $ b \in B $ and an open set $ V \subseteq A $.  We shall denote the set $  M(B,A) $ endowed with point-open topology by $  M(B,A)_{\text{po}} $. In particular, if we let \[ \varepsilon_{b}  : M(B,A)_{\text{po}} \longrightarrow A \ \  \text{be the map given by} \ \  f \longmapsto f_{b}\stackrel{\mathrm{def}}{=} f(b), \] which is the natural projections, then we have a natural homeomorphism $  M(B,A)_{\text{po}}\cong\prod_{b \in B}A_{b} $, where $ \prod_{b \in B}A_{b} $ is endowed with product (Tychonoff) topology cf. \cite[p.~217, Chapter 7]{MR0370454}. As such the map $ \varepsilon_{b} $ is continuous for all $ b \in B $ since for any open subset $ V \subset A $ we have $ \varepsilon_{b}^{-1}(V)=\beta(\{b\},V) $, which is an open subset of $  M(B,A)_{\text{po}} $.\footnote{When $ A $ is a uniform space see \cite[p.~226]{MR0370454} for the definition of the topology of uniform convergence on $ M(B,A) $, and see \cite[p.~229]{MR0370454} for the definition of the topology of compact convergence on $ M(B,A) $.} 
	
	Since the point-open topology on $ M(B,A) $ is coarser than the compact-open topology, the map  $  \text{id} : M(B,A)_{\text{co}} \longrightarrow M(B,A)_{\text{po}} $ is continuous, and for all $ b \in B $, we have a commutative diagram of continuous maps  \[ 
	\begin{tikzcd}[row sep=1.2em , column sep=1.5em]
	M(B,A)_{\text{co}} \arrow{r}[]{\text{id}}\arrow{dr}{}[swap]{\varepsilon_{b}\text{id}} & M(B,A)_{\text{po}}   \arrow{d}[]{\varepsilon_{b}} \\ \ & A.  
	\end{tikzcd} \]
	
	We denote by $ M_{\text{cts}}(B,A) $ the subset of $ M(B,A) $ containing all the continuous maps from $ B\longrightarrow A $. For any subset $ F=F(B,A) \subseteq M(B,A)  $ we let \[  F_{\text{cts}}=F_{\text{cts}}(B,A)\stackrel{\mathrm{def}}{=} F \cap M_{\text{cts}}(B,A). \] We further denote by $ M_{\text{cts}}(B,A)_{\text{po}} $ the set $ M_{\text{cts}}(B,A) $ together with the subspace topology inherited from $ M(B,A)_{\text{po}} $; similarly, we denote by $ M_{\text{cts}}(B,A)_{\text{co}} $ the set $ M_{\text{cts}}(B,A) $ together with the subspace topology inherited from $ M(B,A)_{\text{co}} $.
	
	For the rest of the document in places we shall use the word ''space'' by which we shall always mean a ''topological space'' when the topology under consideration is understood.
	
	We note that the point-open topology on $ M(B,A) $ is completely determined by the topology of $ A $, so by working with this topology we are often ignoring the information offered by the topology of $ B $ i.e., we are assuming $ B $ is a set; therefore if one would like to take into consideration the information offered by the topology of $ B $, then the compact-open topology provides this opportunity. It also happens that these two topologies coincide for many cases; for example, if $ A $ is a Hausdorff space, then one finds that $ M(B,A)_{\text{po}} $ is a Hausdorff space; now in a given situation when $ M(B,A)_{\text{co}} $ is a compact space, then it follows that the map $ \text{id} $ is a homeomorphism, and therefore one has $ M(B,A)_{\text{co}}\cong M(B,A)_{\text{po}} $.
	
	We mainly focus on working with the space $  M(B,A)_{\text{po}} $; however, most of our investigations can also be carried out when replacing the point-open topology with the compact-open topology. We first recall a few basic facts relating to point-open topology.    
	\begin{lemma}\label{L1}
		Let $ B $ be a topological space. Assume $ \varphi : A \longrightarrow A' $ is a continuous map between topological spaces. Then the map \[ \varphi_{\ast} : M(B,A)_{\text{po}} \longrightarrow M(B,A')_{\text{po}} \ \ \text{given by} \ \ f\longmapsto \varphi f \stackrel{\mathrm{def}}{=} \varphi\circ f \ \ \text{is continuous.}\]  In particular, the map $ \varphi_{\ast} $ restricts to a continuous map (also denoted by $ \varphi_{\ast} $) \[ \varphi_{\ast}: M_{\text{cts}}(B,A)_{\text{po}} \longrightarrow M_{\text{cts}} (B,A')_{\text{po}}. \] 
	\end{lemma}
	\begin{proof}
		To show the continuity of $ \varphi_{\ast} $, it suffices to check that the inverse image of an element of the subbase for the point-open topology on $ M (B,A') $ is an open subset of $ M(B,A)_{\text{po}} $. Fix a point $ b \in B $ and an open subset $ V \subseteq A' $. Note since $ \varphi $ is continuous $ \varphi^{-1}(V) \subseteq A $ is an open subset. Let $ \beta(\{b\},V) \subseteq M(B,A')_{\text{po}}  $ be an element of the subbase for the point-open topology on $ M(B,A') $. Then we have \[\varphi_{\ast}^{-1}(\beta(\{b\},V))= \{f \in M(B,A):  \varphi (f(b)) \in V \}; \]
		it follows that \[\varphi_{\ast}^{-1}(\beta(\{b\},V))=\beta(\{b\},\varphi^{-1}(V))\]
		since if $ f \in \varphi_{\ast}^{-1}(\beta(\{b\},V)) $, then $  \varphi (f(b)) \in V $ so $ f(b) \in \varphi^{-1}(V)  $ i.e., $ f \in \beta(\{b\},\varphi^{-1}(V))  $. Conversely, if $ g \in \beta(\{b\},\varphi^{-1}(V)) $, then $ g(b ) \in \varphi^{-1}(V) $, and so \[ \varphi (g(b ))=\varphi g(b) \in \varphi(\varphi^{-1}(V)) \subseteq V  \] i.e., $ \varphi g=\varphi_{\ast}(g)\in \beta(\{b\},V)  $ so $ g \in \varphi_{\ast}^{-1}(\beta(\{b\},V))  $. 
		
		Now since $ \beta(\{b\},\varphi^{-1}(V)) \subseteq M(B,A)_{\text{po}} $ is an element of the subbase for the topology, which follows from continuity of $ \varphi $, we conclude that $ \varphi_{\ast} $ is a continuous map. In particular, since the composition of two continuous maps is a continuous map; if $ f \in M_{\text{cts}}(B,A)_{\text{po}} $, then $ \varphi f \in M_{\text{cts}} (B,A')_{\text{po}} $; now the second statement follows since restriction of a continuous map to a subspace is a continuous map cf. \cite[p.~85]{MR0370454}, and this proves the lemma. 
	\end{proof}
	\begin{lemma}\label{L2}
		Let $ A $ be a topological space. Assume $ \psi: B\longrightarrow B' $ is a continuous map between topological spaces. Then the map \[ \psi^{\ast}: M(B',A)_{\text{po}} \longrightarrow M (B,A)_{\text{po}} \ \ \text{given by} \ \ f\longmapsto f\psi \stackrel{\mathrm{def}}{=} f\circ \psi  \ \ \text{is continuous.} \]  In particular, the map $ \psi^{\ast} $ restricts to a continuous map (also denoted by $ \psi^{\ast} $) \[ \psi^{\ast}: M_{\text{cts}}(B',A)_{\text{po}} \longrightarrow M_{\text{cts}} (B,A)_{\text{po}}. \]   
	\end{lemma}
	\begin{proof}
		Similarly to Lemma \ref{L1}, we check the continuity of $ \psi^{\ast} $ on the elements of the subbase for the topology of $ M(B,A)_{\text{po}} $. Fix a point $ b \in B $ and an open subset $ V \subseteq A $. Let $ \beta(\{b\},V) \subseteq M(B,A)_{\text{po}}  $ be an element of the subbase for the point-open topology on $ M(B,A) $. Then we have \[ \psi^{\ast} \ ^{-1}(\beta(\{b\},V))= \{f \in M(B',A): f(\psi(b))\in V \};  \]
		it follows that
		\[ \psi^{\ast} \ ^{-1}(\beta(\{b\},V))= \beta(\{\psi(b)\},V) \]
		since if $ f \in \psi^{\ast} \ ^{-1}(\beta(\{b\},V))  $, then $ f\psi(b) \in V $ so $ f \in \beta(\{\psi(b)\},V) $. Conversely, if $ g \in \beta(\{\psi(b)\},V) $, then $ g\psi(b) \in V $, and so $ g\psi=\psi^{\ast}(g) \in \beta(\{b\},V)  $; hence, $ g \in \psi^{\ast} \ ^{-1}(\beta(\{b\},V)) $. 
		
		Now since $ \beta(\{\psi(b)\},V) \subseteq M(B',A)_{\text{po}} $ is an element of the subbase for the topology, we conclude that $ \psi^{\ast} $ is a continuous map. In particular, since the composition two continuous maps is a continuous map, if $ f \in M_{\text{cts}}(B',A)_{\text{po}} $, then $ f\psi \in M_{\text{cts}} (B,A)_{\text{po}} $. Now the second statement follows since restriction of a continuous map to a subspace is a continuous map cf. \cite[p.~85]{MR0370454}, and this proves the lemma.
	\end{proof}
	\begin{lemma}\label{L3}
		Suppose $ A $ and $ B $ are topological spaces. 
		\begin{enumerate}[i)]
			\item If $ A $ is a Hausdorff space, then the spaces $ M(B,A)_{\text{po}} $ and $ M_{\text{cts}}(B,A)_{\text{po}} $ are Hausdorff.
			\item If $ A $ is a compact space, then $ M(B,A)_{\text{po}} $ is a compact topological space.
		\end{enumerate} 
	\end{lemma}
	\begin{proof}
		Note we have a natural identification $  M(B,A)_{\text{po}}\cong\prod_{b \in B}A_{b} $, and so $ M(B,A)_{\text{po}} $ is a product of Hausdorff spaces; therefore it is a Hausdorff space cf. \cite[p.~ 87, 13.8 Theorem. b)]{MR2048350}. Now since $ M_{\text{cts}}(B,A)_{\text{po}} \subseteq M(B,A)_{\text{po}} $ is a subspace of a Hausdorff space, we conclude that $ M_{\text{cts}}(B,A)_{\text{po}} $ is a Hausdorff space cf. \cite[p.~87, 13.8 Theorem. a)]{MR2048350}, and this proves i). 
		
		If $ A $ is a compact space, then since $  M(B,A)_{\text{po}}\cong\prod_{b \in B}A_{b} $ is a product of compact spaces, we conclude that $  M(B,A)_{\text{po}} $ is a compact space cf. \cite[p.~120, 17.8 Theorem (Tychonoff)]{MR2048350}, and this proves ii). 
	\end{proof}
	
	We shall fix a directed poset $ R $ and assume $ \{A_{r}, \varphi_{rs}, R\} $ is an inverse system of topological spaces see \cite[p.~1, for definition]{MR2599132} for the rest of the article
	\begin{lemma}\label{L4}
		Let $ A $ and $ B $ be topological spaces. Assume $ A $ has a presentation $ A= \varprojlim_{\stackrel{ \ }{r\in R}}A_{r} $. Denote by $ \varphi_{r}: A\longrightarrow A_{r} $ the natural projection. Then there exists a continuous bijection \[ \theta : M_{\text{cts}}(B,A)_{\text{po}} \longrightarrow \varprojlim_{\stackrel{ \ }{r\in R}} M_{\text{cts}}(B,A_{r})_{\text{po}} \ \ \text{given by} \ \ f\longmapsto(\varphi_{r\ast}(f)). \] In particular, if $ A $ is a compact space and $ A_{r} $ is Hausdorff for each $ r \in R $, then $ \theta $ is a homeomorphism. 
	\end{lemma}
	\begin{proof}
		By Lemma \ref{L1} the continuous transition maps $ \varphi_{rt}: A_{r}\longrightarrow A_{t} $ for $ r,t \in R $ and $ r\geq t $ induce a commutative diagram of continuous maps 
		\[ 
		\begin{tikzcd}[row sep=1.2em , column sep=2em]
		M_{\text{cts}}(B,A_{r})_{\text{po}} \arrow[hook]{d}{}\arrow{r}{\varphi_{rt\ast}}[swap]{} & M_{\text{cts}}(B,A_{t})_{\text{po}} \arrow[hook]{d}{} \\ M(B,A_{r})_{\text{po}} \arrow{r}{\varphi_{rt\ast}}[swap]{} & M(B,A_{t})_{\text{po}};  
		\end{tikzcd} \]
		therefore we see that the sets $ \{M_{\text{cts}}(B,A_{r}), \varphi_{rt\ast}, R \} $ and $ \{M(B,A_{r}), \varphi_{rt\ast}, R \} $ are inverse system of topological spaces; hence the inverse limits exists and is unique up to isomorphism cf. \cite[p.~2, Proposition 1.1.1]{MR2599132}. In particular, by the universal property of the inverse limit, we obtain a unique continuous commutative diagram
		\[ 
		\begin{tikzcd}[row sep=1.2em , column sep=2em]
		M_{\text{cts}}(B,A)_{\text{po}} \arrow[hook]{d}{}\arrow{r}{\theta}[swap]{} & \varprojlim_{\stackrel{ \ }{r\in R}} M_{\text{cts}}(B,A_{r})_{\text{po}} \arrow[hook]{d}{} \\ M(B,A)_{\text{po}} \arrow{r}{\theta}[swap]{} & \varprojlim_{\stackrel{ \ }{r\in R}} M(B,A_{r})_{\text{po}},  
		\end{tikzcd} \]
		it is clear that $ \theta $ on the bottom horizontal arrow is a continuous bijection; also one has $ \theta(M_{\text{cts}}(B,A)_{\text{po}})=\varprojlim_{\stackrel{ \ }{r\in R}} M_{\text{cts}}(B,A_{r})_{\text{po}}  $ since given an element $ (f_{r}) \in \varprojlim_{\stackrel{ \ }{r\in R}} M_{\text{cts}}(B,A_{r})_{\text{po}} $, then by the universal property of the inverse limit $ \varprojlim_{\stackrel{ \ }{r\in R}}A_{r} $ there exists a unique continuous map $ f : B \longrightarrow A $ such that $ \varphi_{r\ast}(f)=f_{r} $ for all $ r \in R $ cf. \cite[p.~2]{MR2599132}.
		
		In particular, if $ A $ is a compact space, then by Lemma \ref{L3}, ii), $  M(B,A)_{\text{po}} $ is a compact space; further if $ A_{r} $ is Hausdorff for each $ r \in R $, then by Lemma \ref{L3}, i), the space $ M_{\text{cts}}(B,A_{r}) $ is Hausdorff for each $ r \in R $; therefore \[ \varprojlim_{\stackrel{ \ }{r\in R}} M(B,A_{r})_{\text{po}} \subseteq \prod_{r \in R} M(B,A_{r})_{\text{po}} \] is a Hausdorff topological space cf. \cite[p.~87, 13.8 Theorem. a) \& b)]{MR2048350}; thus in such case $ \theta $ is a continuous bijection from a compact space to a Hausdorff space, which implies that $ \theta $ is a homeomorphism cf. \cite[p.~123, 17.14 Theorem]{MR2048350}, this proves the lemma.      
	\end{proof}
	\begin{definition}[Evenly Continuous Sets]\label{D1}
		We say a subset $ F \subseteq M(B,A) $ is \textit{evenly continuous} if for each $ b \in B $, each $ a \in A $, and each neighbourhood $ U $ of $ a $ there exists a neighbourhood $ V $ of $ b $ and a neighbourhood $ W $ of $ a $ such that for all $ f \in F $ we have $ f(V) \subseteq U $ whenever $ f(b) \in W $. cf. \cite[p.~235]{MR0370454}. We say $ A $ is \textit{evenly continuous with respect to a topological space $ B $} if the set $ M_{\text{cts}}(B,A) $ is evenly continuous. 
	\end{definition}
	
	Two obvious examples of evenly continuous sets are as follows. For two topological spaces $ A $, and $ B $, denote by $ M(B^{0},A) \subseteq M(B,A) $ the of all \textit{constant maps} from $ B\longrightarrow A $. Then $ M(B^{0},A) $ is an evenly continuous set since let $ b \in B $, $ a \in A $, and $ U $ be a neighbourhood of $ a $. Then $ B $ is a neighbourhood of $ b $ and $ U $ is a neighbourhood of $ a $ such that for all $ f \in M^{0}(B,A) $ we have $ f(B)=a_{f} \in U $ whenever $ f(b)=a_{f} \in U $, and this show that $ M(B^{0},A) $ is evenly continuous; also if $ A $ and $ B $ are discrete spaces, then $ A $ is evenly continuous with respect to $ B $ since let $ b \in B $, $ a \in A $ and $ U $ a neighbourhood of $ a $, then $ \{b\} $ is a neighbourhood of $ b $ and $ U $ is a neighbourhood of $ a $ such that $ f(\{b\})\subseteq U $ whenever $ f(b)\in U $ for all $ f \in M_{\text{cts}}(B,A) $.
	
	\begin{lemma}\label{L5}
		Suppose $ A $ and $ B $ are compact and Hausdorff spaces. If $ A $ is an evenly continuous space with respect to $ B $, then $ M_{\text{cts}}(B,A)_{\text{po}} $ is a compact space. In particular, if $ A $ has a presentation $ A= \varprojlim_{\stackrel{ \ }{r\in R}}A_{r} $, where $ A_{r} $ is a compact, Hausdorff, which is an evenly continuous space with respect to $ B $ for each $ r \in R $, then $ M_{\text{cts}}(B,A)_{\text{po}} $ is a compact space.
	\end{lemma}
	\begin{proof}
		If $ A $ and $ B $ are compact and Hausdorff spaces, then $ A $ and $ B $ are $ T_{4} $-spaces cf. \cite[p.~121, 17.1 Theorem]{MR2048350}; therefore $ A $ and $ B $ are regular spaces; in particular $ A $ and $ B $ are also locally compact cf. \cite[p.~130, 18.2 Theorem]{MR2048350}. Now if $ A $ is evenly continuous with respect to $ B $, then $ M_{\text{cts}}(B,A)_{\text{co}} $ is a compact space cf. \cite[p.~236, 21 ASCOLI THEOREM]{MR0370454}. It also follows from Lemma \ref{L3}, i), that $ M_{\text{cts}}(B,A)_{\text{po}} $ is a Hausdorff space. Now the map $ \text{Id}: M_{\text{cts}}(B,A)_{\text{co}} \longrightarrow M_{\text{cts}}(B,A)_{\text{po}} $ is a continuous bijection from a compact space to a Hausdorff space, and therefore it is a homeomorphism cf. \cite[p.~123, 17.14 Theorem]{MR2048350}, which implies that $ M_{\text{cts}}(B,A)_{\text{po}} $ is a compact space. 
		
		In particular, if $ A $ has a presentation $ A= \varprojlim_{\stackrel{ \ }{r\in R}}A_{r} $, where $ A_{r} $ is a compact, Hausdorff and evenly continuous space with respect to $ B $ for each $ r \in R $, then by above $ M_{\text{cts}}(B,A_{r})_{\text{po}} $ is a compact space and by Lemma \ref{L3}, i), $ M_{\text{cts}}(B,A_{r})_{\text{po}} $ is a Hausdorff space; further by Lemma \ref{L4} we have a homeomorphism  \[ \theta : M_{\text{cts}}(B,A)_{\text{po}} \longrightarrow \varprojlim_{\stackrel{ \ }{r\in R}} M_{\text{cts}}(B,A_{r})_{\text{po}} \ \ \text{given by} \ \ f\longmapsto(\varphi_{r,\ast}(f)). \] Now since $ \varprojlim_{\stackrel{ \ }{r\in R}} M_{\text{cts}}(B,A_{r})_{\text{po}}  $ is a closed subset of $ \prod_{r \in R}M_{\text{cts}}(B,A_{r})_{\text{po}}  $ cf. \cite[p.~3, Lemma 1.1.2]{MR2599132}, and $ \prod_{r \in R}M_{\text{cts}}(B,A_{r})_{\text{po}}  $ is a compact space  cf. \cite[p.~120, 17.8 Theorem (Tychonoff)]{MR2048350}, we have that $ \varprojlim_{\stackrel{ \ }{r\in R}} M_{\text{cts}}(B,A_{r})_{\text{po}}  $ is a compact space cf. \cite[p.~120, 17.5 Theorem. a)]{MR2048350}; therefore we conclude that $ M_{\text{cts}}(B,A)_{\text{po}} $ is a  compact space, which proves the lemma. 
	\end{proof} 
	\subsection{$ G $-Groups}\label{SB2}
	We shall be concerned with a fixed topological group $ G $ for the rest of the document. We denote the continuous inversion and multiplication of a topological group $ H $ by $ \iota_{H} : H\longrightarrow H $ and $ m_{H}: H \times H \longrightarrow H $ respectively.
	\begin{definition}[(Continuous) $ G $-Group, $ G $-Space, and $ G $-Map]\label{D2}\footnote{Compare with \cite[p.~7, (1.7.1) Definition]{MR2392026} \& \cite[p.~45, 5.1]{MR1867431}.} A \textit{(topological) $ G $-group} is a Hausdorff topological group $ A $ on which $ G $ acts (on the left) continuously and compatibly with the group operation of $ A $. In other words, $ A $ is a \textit{$ G $-group} if $ A $ is a Hausdorff topological group together with a continuous map $ ac: G\times A \longrightarrow A $ (where $ G\times A $ is endowed with product topology), and we shall write $ \ ^{s}x \stackrel{\mathrm{def}}{=} ac(s,x) $, such that the following holds:
		\begin{enumerate}[i)]
			\item $ ac(1,x)=x $ and $ ac(m_{G}(s,t),x) =ac(s,ac(t,x)) $ i.e., $ ^{1}x=x $ and $ ^{st}x= \ ^{s}(^{t}x) $, for all $ s,t \in G $ and $ x \in A $.
			\item $ ac(s,m_{A}(x,y))=m_{A}(ac(s,x),ac(s,y)) $ i.e., $ ^{s}(xy)= \ ^{s}x^{s}y $, for all $ s \in G $ and $ x,y \in A $.
		\end{enumerate} 
		
		A Hausdorff topological space $ A' $ together with a continuous map $ ac: G\times A' \longrightarrow A'  $ such that condition i) above holds is known as a \textit{(topological) $ G $-space}.
		
		A \textit{$ G $-group homomorphism} between two $ G $-groups $ A $ and $ A' $ is a continuous group homomorphism $ \varphi : A\longrightarrow A' $ such that $ \varphi $ commutes with the action of $ G $, i.e., $ \varphi(ac(s,x))=ac(s,\varphi(x)) $, for all $ s \in G $ and $ x \in A $. Similarly, a \textit{map of $ G $-spaces} is a continuous map which commutes with the action of $ G $.
	\end{definition}
	\begin{lemma}\label{L6}
		Let $ A $ be a $ G $-group. Then the map \[ \rho : G \longrightarrow M(A,A)_{\text{po}} \ \  \text{given by} \ \  s\longmapsto \rho(s)\stackrel{\mathrm{def}}{=} ac_{\mid_{(\{s\}\times A)}} \] is continuous. In particular, one has $ \Ima \rho \subseteq \text{Aut}_{\text{cts}}(A)_{\text{po}} $ as a subgroup.
	\end{lemma}
	\begin{proof}
		To show $ \rho  $ is continuous, fix a point $ x \in A $ and an open subset $ V \subseteq A $. Let $ \beta(\{x\},V) \subseteq M(A,A)_{\text{po}}  $ be an element of the subbase for the point-open topology on $ M(A,A) $. Let $ s \in \rho^{-1}(\beta(\{x\},V)) $. Then $ ac_{\mid_{(\{s\}\times A)}}(x)=ac(s,x)  \in V  $, so $ (s,x) \in ac^{-1}(V)  $. By continuity of $ ac $ the set $ ac^{-1}(V) $ is an open subset of $ G \times A $; therefore there exist open subsets $ U_{1} \subseteq G $ and $ V_{1} \subset A $ with $ (s , x) \subseteq U_{1} \times V_{1} \subseteq  ac^{-1}(V)  $. Now $ s \in U_{1}   $, and if $ t \in U_{1}  $, then $ \rho(t)(x)=ac(t,x) \in V $, so $\rho(t) \in \beta(\{x\},V)  $ i.e., $ t \in \rho^{-1}(\beta(\{x\},V)) $, therefore $  U_{1} \subseteq \rho^{-1}(\beta(\{x\},V))  $. This shows $ \rho^{-1}(\beta(\{x\},V))  $ is an open subset of $ G $; hence $ \rho $ is a continuous map. In particular, since the action of $ G $ respects the multiplication of $ A $ i.e., since $ ac_{\mid_{(\{s\}\times A)}} $ is a continuous group automorphism of $ A $ (upon composing with the homeomorphism $ \{s\}\times A\cong A $) for every $ s \in G $, and $ ac_{\mid_{(\{st\}\times A)}}=ac_{\mid_{(\{s\}\times A)}} ac_{\mid_{(\{t\}\times A)}}  $, we must also have $ \Ima \rho \subseteq \text{Aut}_{\text{cts}}(A)_{\text{po}} $. This shows that $ \Ima \rho \subseteq \text{Aut}_{\text{cts}}(A)_{\text{po}} $ as a subgroup, which proves the lemma.
	\end{proof}		
	\subsection{Continuous Cocycles and Torsors}\label{SB3}
	 We define the set of continuous cocycles endowed with point-open topology and investigate its topological properties. We also define the set of continuous torsors and briefly discuss their well-known relationship to continuous cocycles.	
	\begin{definition}[Cocycles and Continuous Cocycle Pointed Spaces]\label{D3}\footnote{Compare with \cite[p.~45, 5.1]{MR1867431}} Let $ A $ be a $ G $-group. Define the set of \textit{(1-)cocycles of $ G $ with values in} $ A $ by 
		\[Z^{1}(G,A)_{\text{po}}\stackrel{\mathrm{def}}{=}\{a \in M(G,A)_{\text{po}}: a_{st}=a_{s} \ ^{s}a_{t} \ \text{for all} \ s,t \in G \}.\]
		We Let $ Z^{1}_{\text{cts}}(G,A)_{\text{po}}\stackrel{\mathrm{def}}{=} Z^{1}(G,A)_{\text{po}}\cap M_{\text{cts}}(G,A)_{\text{po}} $.
	\end{definition}
	The sets $ Z^{1}(G,A)_{\text{po}} $, and $ Z^{1}_{\text{cts}}(G,A)_{\text{po}} $ are topological pointed sets, whose topologies are the subspace topology inherited from $ M(G,A)_{\text{po}} $, and their distinguished point is given by the trivial cocycle $ a_{1} $, which is the constant map $ a_{1,s}=1 $ for all $ s \in G $.  
	\begin{remark}\label{R1}
		Given a $ G $-group homomorphism $ \varphi : A\longrightarrow A' $ between $ G $-groups $ A $ and $ A' $ as defined in Definition \ref{D3}, it follows from Lemma \ref{L1}, and the fact that $ \varphi $ commutes with the action of $ G $ i.e., $ \varphi(\ ^{s}x)= \ ^{s}\varphi(x) $ for all $ s \in G $ and $ x \in A $, that $ \varphi $ induces a continuous maps \[ \varphi_{\ast} : Z^{1}(G,A)_{\text{po}}\longrightarrow Z^{1}(G,A')_{\text{po}},\] which restricts to a continuous map (also denoted by $ \varphi_{\ast} $) \[ \varphi_{\ast} : Z^{1}_{\text{cts}}(G,A)_{\text{po}}\longrightarrow Z^{1}_{\text{cts}}(G,A')_{\text{po}}.\] 
		
		In addition, if $ \psi : H\longrightarrow G $ is a continuous group homomorphism and $ A $ is a $ G $-group, then $ A $ can be regraded as a $ H $-group, and using Lemma \ref{L2}, $ \psi $ induces a continuous map \[  \psi^{\ast} : Z^{1}(G,A)_{\text{po}}\longrightarrow Z^{1}(H,A)_{\text{po}},\] which restricts to a continuous map (also denoted by $ \psi^{\ast} $) \[ \psi^{\ast} : Z^{1}_{\text{cts}}(G,A)_{\text{po}}\longrightarrow Z^{1}_{\text{cts}}(H,A)_{\text{po}}.\]      
	\end{remark}
	
	Let  $ A $ be a $ G $-group. We briefly recall the well-known relationship between continuous cocycles of $ G $ in $ A $ and continuous $ G $-torsors under $ A $. 
	\begin{definition}[Continuous $ G $-Torsor]\label{D4}\footnote{Compare with \cite[p.~46, 5.2]{MR1867431}.} Let $ A $ be a $ G $-group. Then a \textit{continuous $ G $-torsor under}  $ A $ is a topological space $ P $, which is a left $ G $-space with action \[ ac_{G}:G\times P\longrightarrow P \ \ \text{given by} \ \ (s,p)\longmapsto ac_{G}(s,p)\stackrel{\mathrm{def}}{=} \ ^{s}p \] and a right $ A $-space with action \[ ac_{A}:P\times A\longrightarrow P \ \ \text{given by} \ \ (p,x)\longmapsto ac_{A}(p,x) \stackrel{\mathrm{def}}{=} p.x  \] such that the map \[\phi_{A}: P\times A\longrightarrow  P \times P \ \ \text{given by} \ \ (p,x)\longmapsto (p,p.x) \]
		is a homeomorphism of $ G $-spaces.
		
		We denote by $ \text{Tors}^{G}_{\text{cts}}(A) $ the set of continuous $ G $-torsors under $ A $. A \textit{map of continuous $ G $-torsors under} $ A $  is a continuous map which commutes with the actions of $ G $ and $ A $.  
	\end{definition}
	
	Note if $ P $ is a continuous $ G $-torsor under $ A $ as in Definition \ref{D4}, then the map $ \phi_{A} $ is a $ G $-space map if and only if $ ac_{A} $ is a map of $ G $-spaces; since we have  \[ \phi_{A}(^{s}p, \ ^{s}x)=(^{s}p, \ ^{s}p.^{s}x)=(^{s}p, \ ^{s}(p.x)) \] for all $ s \in G $, $ p \in P $ and $ x \in A $ if and only if $ ^{s}p. ^{s}x= \ ^{s}(p.x) $ for all $ s \in G $, $ p \in P $ and $ x \in A $.
	\begin{remark}[Correspondence between continuous cocycles and torsors]\label{R2}
		Let $ a \in Z^{1}_{\text{cts}}(G,A)_{\text{po}} $. Then one can construct a continuous $ G $-torsor under $ A $ denoted by $ _{a}P $ as follows. The underlying set of $ _{a}P $ is $ A $ on which $ G $ acts on the left by \[ (s,p)\longmapsto \ ^{s_{'}}p\stackrel{\mathrm{def}}{=} a_{s} \ ^{s}p \stackrel{\mathrm{def}}{=} m_{A}(a_{s}, \ ^{s}p),\] and $ A $ acts on the right by \[ (p,x)\longmapsto p.x\stackrel{\mathrm{def}}{=} m_{A}(p,x).\]
		Therefore, we have $ ^{s_{'}}(p.x)=a_{s} \ ^{s}(px)=a_{s} \ ^{s}p \ ^{s}x= \ ^{s_{'}}p. \ ^{s}x $, i.e., the action of $ A $ is a $ G $-space map. One verifies immediately that $ _{a}P $ is a continuous $ G $-torsor. In particular, for $ a,b \in Z^{1}_{\text{cts}}(G,A)_{\text{po}} $ we have $ _{a}P= \  _{b}P$ if and only if $ a=b $. This gives a well-defined injective map of sets \[\mathcal{T}: Z^{1}_{\text{cts}}(G,A)_{\text{po}}\longrightarrow \text{Tors}^{G}_{\text{cts}}(A). \] 
		
		Now let $ P $ be a $ G $-torsor under $ A $ as defined in Definition \ref{D4}. Choose a point $ p_{0} \in P $. Then we obtain a continuous map $ a: G \longrightarrow A $, which makes the following diagram commutative
		\[ 
		\begin{tikzcd}[row sep=1.1em , column sep=5em]
		G\times P  \arrow[]{r}[]{(s,p)\mapsto (p,\ ^{s}p)} & P\times P \arrow{r}[]{\phi_{A}^{-1}}  & P\times A\arrow{d}[]{} \\ G\cong G\times \{p_{0}\}\arrow{u}[]{} \arrow{rr}[]{a} && A,
		\end{tikzcd} \]
		where the left vertical arrow is the natural inclusion and the right vertical arrow is the projection to the second factor. Therefore, for any $ s \in G $, $ a_{s} \in A $ is the unique element satisfying $ ^{s}p_{0}=p_{0}.a_{s} $ (this is possible since $ \phi_{A} $ is a bijection). In particular, since $ \phi_{A} $ is a map of $ G $-spaces and for each $ s,t \in G $, we have that $ a_{st} \in A $ is the unique element satisfying $ ^{st}p_{0}=p_{0}.a_{st} $, we find \[ ^{st}p_{0}= \ ^{s}(^{t}p_{0})= \ ^{s}(p_{0}.a_{t})= \ ^{s}p_{0}. ^{s}a_{t}=p_{0}. a_{s}. \ ^{s}a_{t}=p_{0}.(a_{s} \ ^{s}a_{t});\]
		since $ \phi $ is injective, we have $ a_{st}=a_{s} \ ^{s}a_{t} $ for all $ s,t \in G $; from this it follows that $ a $ is a continuous cocycle.
		
		Furthermore, by restriction of $ \phi_{A} $ to the subset $ \{p_{0}\}\times A $, we obtain a canonical homeomorphism of $ G $,$ A $-spaces 
		\[ A\cong \{p_{0}\}\times A\longrightarrow  \{p_{0}\} \times P\cong P \ \ \text{given by} \ \ x\longmapsto p_{0}.x, \] 
		where $ A $ acts on itself by right multiplication, and $ G $ acts on $ A $ by $ ^{s_{'}}x= a_{s} \ ^{s}x $ i.e., $ P \cong  \ _{a} P $ as $ G $,$ A $-spaces. This shows that the map $ \mathcal{T} $ is a bijection of sets. Therefore, we may introduce a topology on $ \text{Tors}^{G}_{\text{cts}}(A) $ by declaring a subset to be open if it is the image of an open subset of $ Z^{1}_{\text{cts}}(G,A)_{\text{po}} $ under $ \mathcal{T} $. In such case  $ \mathcal{T} $ will be a homeomorphism, and we shall denote by $ \text{Tors}^{G}_{\text{cts}}(A)_{\text{po}} $ the set $ \text{Tors}^{G}_{\text{cts}}(A) $ with the topology imposed by the bijection $  \mathcal{T} $. 
	\end{remark}
	\begin{lemma}\label{L7}
		Let $ A $ be a $ G $-group. Then $ Z^{1}(G,A)_{\text{po}} \subseteq M(G,A)_{\text{po}} $ is a closed subset; so $  Z_{\text{cts}}^{1}(G,A)_{\text{po}} \subseteq M_{\text{cts}}(G,A)_{\text{po}}  $ is a closed subset. In particular, if $ G $ is a compact and Hausdorff space, and $ A $ is a compact $ G $-group evenly continuous with respect to $ G $, then $  Z_{\text{cts}}^{1}(G,A)_{\text{po}} $ is a compact and Hausdorff space. 
	\end{lemma}
	\begin{proof}
		To show $ Z^{1}(G,A)_{\text{po}} \subseteq M(G,A)_{\text{po}} $ is a closed subset, let $ s,t \in G $ and define a map $ F_{s,t}: M(G,A)_{\text{po}}\longrightarrow A  $ by \[ F_{s,t}(f)\stackrel{\mathrm{def}}{=} f_{s} \ ^{s}f_{t}f_{st}^{-1}\stackrel{\mathrm{def}}{=} m_{A}(f_{s},m_{A}(ac(s,f_{t}),\iota_{A}(f_{st}))).  \]
		We claim that the map $ F_{s,t} $ is continuous. We shall check the continuity of $ F_{s,t}  $ directly. 
		
		Let $ V \subseteq A $ be an open subset and $ f \in F_{s,t}^{-1}(V) $. Then \[ m_{A}(f_{s},m_{A}(ac(s,f_{t}),\iota_{A}(f_{st}))) \in V,\] so $  (f_{s},m_{A}(ac(s,f_{t}),\iota_{A}(f_{st}))) \in m_{A}^{-1}(V) $, and by continuity of $ m_{A} $ the subset $ m_{A}^{-1}(V) $ is open in $ A\times A $; therefore there exist open subsets $ U_{1} ,V_{1} \subseteq A $ such that \[(f_{s}, m_{A}(ac(s,f_{t}),\iota_{A}(f_{st})))  \in U_{1}\times V_{1} \subseteq m_{A}^{-1}(V),  \]
		so $ f \in \beta(\{s\},U_{1}) $, and $ m_{A}(ac(s,f_{t}),\iota_{A}(f_{st})) \in  V_{1}  $, hence $  (ac(s,f_{t}),\iota_{A}(f_{st})) \in m_{A}^{-1}(V_{1}) $; since $ m_{A} $ is continuous, and $ V_{1} \subseteq A $ is an open subset, we can find open subsets $ U_{2} ,V_{2} \subseteq A $ such that \[ (ac(s,f_{t}),\iota_{A}(f_{st}))   \in U_{2}\times V_{2} \subseteq m_{A}^{-1}(V_{1}),  \]
		i.e., $ (s,f_{t}) \in  ac^{-1}(U_{2})  $ and $ f_{st} \in \iota_{A} (V_{2})  $; therefore $ (s,f_{t}) \in  ac^{-1}(U_{2})  $ and $ f \in \beta(\{st\},\iota_{A} (V_{2})) $. 
		
		Now since $ ac $ is a continuous map, $ ac^{-1}(U_{2}) $ is an open subset of $ G\times A $, so there exist open subsets $ U_{3} \subseteq G $ and $ V_{3}\subseteq A $ such that \[(s,f_{t}) \in U_{3} \times V_{3}  \subseteq  ac^{-1}(U_{2}) \ , \] 
		hence $ s \in U_{3} $ and $ f \in \beta(\{t\},V_{3}) $. 
		Now \[ f \in \beta(\{s\},U_{1})\cap  \beta(\{st\},\iota_{A} (V_{2})) \cap \beta(\{t\},V_{3}), \] which is an open subset of $ M(G,A)_{\text{po}} $, and if \[ g \in \beta(\{s\},U_{1})\cap\beta(\{st\},\iota_{A} (V_{2}))  \cap \beta(\{t\},V_{3}),  \] then $ g_{s} \in U_{1} $, $ g_{st}^{-1} \in V_{2}  $, and $ g_{t} \in V_{3}  $, so $ ^{s}g_{t} \in U_{2} $ and so $ ^{s}g_{t} g_{st}^{-1} \in V_{1}  $ and $ g_{s} \ ^{s}g_{t} g_{st}^{-1} \in V $, this implies that $ F_{s,t}(g) \in V $, hence $ g \in F_{s,t}^{-1}(V) $. From this it follows that $ F_{s,t}^{-1}(V) $ is an open subset of $ M(G,A)_{\text{po}} $, and so $ F_{s,t} $ is a continuous map; since $ s,t $ where arbitrary, $ F_{s,t} $ is a continuous map for all $ s,t \in G $.  
		
		Now since $ A $ is a Hausdorff space, the subset $ \{1\} \subseteq A $ is closed, so $ F_{s,t}^{-1}(\{1\}) \subseteq M(G,A)_{\text{po}}  $ is a closed subset by the continuity of $ F_{s,t} $. In particular, we have 
		\[Z^{1}(G,A)_{\text{po}}=\bigcap_{s,t \in G}F_{s,t}^{-1}(\{1\}),\] 
		which is an intersection of closed subsets of $ M(G,A)_{\text{po}} $, which implies that $ Z^{1}(G,A)_{\text{po}} $ is a closed subset of $ M(G,A)_{\text{po}} $; so $ Z^{1}_{\text{cts}}(G,A)_{\text{po}}\stackrel{\mathrm{def}}{=} Z^{1}(G,A)_{\text{po}}\cap M_{\text{cts}}(G,A)_{\text{po}}  $ is a closed subset of $ M_{\text{cts}}(G,A)_{\text{po}} $. Furthermore, recall $ A $ is a Hausdorff space by Definition \ref{D2}, so by Lemma \ref{L3} i), $ M_{\text{cts}}(G,A)_{\text{po}} $ is a Hausdorff space, which implies that $ Z^{1}_{\text{cts}}(G,A)_{\text{po}} $ is a Hausdorff space cf. \cite[p.~87, 13.8 Theorem. a)]{MR2048350}. 
		
		Now if $ G $ is compact and Hausdorff space, and $ A $ is a compact $ G $-group, which is evenly continuous with respect to $ G $, then $ M_{\text{cts}}(G,A)_{\text{po}} $ is a compact space by Lemma \ref{L5}, and since $ Z^{1}_{\text{cts}}(G,A)_{\text{po}} $ is closed subset of a compact space, we conclude that $ Z^{1}_{\text{cts}}(G,A)_{\text{po}} $ is also a compact space cf. \cite[p.~119, 17.14 Theorem. a)]{MR2048350}, which proves the lemma.   
	\end{proof}
	\begin{lemma}\label{L8}
		Let $ A $ be a $ G $-group and $ B $ any topological space. Then $  M_{\text{cts}} (B,A)_{\text{po}} $ is a $ G $-group.
	\end{lemma}
	\begin{proof}
		Note $ A $ is a Hausdorff space, so $  M_{\text{cts}}(B,A)_{\text{po}} $ is a Hausdorff space by Lemma \ref{L3}, i). Since $ m_{A} $, $ \iota_{A} $ and $ ac $ are continuous maps, the maps 
		\[ m_{\text{po}}: M_{\text{cts}} (B,A)_{\text{po}} \times  M_{\text{cts}} (B,A)_{\text{po}} \longrightarrow M_{\text{cts}} (B,A)_{\text{po}} \]\[ \text{given by} \ (f,g)\longmapsto m_{A}(f,g),  \]
		\[ \iota_{\text{po}}:M_{\text{cts}} (B,A)_{\text{po}}  \longrightarrow M_{\text{cts}} (B,A)_{\text{po}} \ \ \text{given by} \ \  f\longmapsto \iota_{A}(f), \ \text{and}  \] 
		\[ ac_{\text{po}}: G\times M_{\text{cts}} (B,A)_{\text{po}}  \longrightarrow M_{\text{cts}} (B,A)_{\text{po}} \ \ \text{given by} \ \  (s,f)\longmapsto ac(s,f)=ac_{\mid_{\{s\}\times A}}f  \] 
		are well-defined. We shall directly check each of the maps above is continuous. Then checking the maps $ m_{\text{po}} $, $ \iota_{\text{po}} $, and $ ac_{\text{po}} $ make $  M_{\text{cts}} (B,A)_{\text{po}} $ a group on which $ G $ acts compatibly with the group operation follows from properties of $ m_{A} $, $ \iota_{A} $, and $ ac $. Note the continuity of $ \iota_{\text{po}} $ follows immediately from Lemma \ref{L1}. 
		
		Now fix a point $ b \in B $ and an open set $ V \subseteq A $. Let $ \beta_{\text{cts}}(\{b\},V) \subseteq M_{\text{cts}}(B,A)_{\text{po}}  $ be an element of the subbase for the point-open topology on $ M_{\text{cts}}(B,A)_{\text{po}}  $. In order to show $  m_{\text{po}} $ is continuous, we shall show that every element of $ m_{\text{po}}^{-1}(\beta_{\text{cts}}(\{b\},V)) $ has an open neighbourhood contained in $ m_{\text{po}}^{-1}(\beta_{\text{cts}}(\{b\},V))  $. Let $ (f,g) \in m_{\text{po}}^{-1}(\beta_{\text{cts}}(\{b\},V))  $. Then $ m_{A}(f(b),g(b)) \in V $ i.e., $ (f(b),g(b)) \in m_{A}^{-1}(V) $. Since $ m_{A} $ is continuous, $ m_{A}^{-1}(V) \subseteq A\times A $ is an open subset; therefore there exist open subsets $ U_{1},V_{1} \subseteq A $ with \[ (f(b),g(b)) \in U_{1} \times V_{1} \subseteq  m_{A}^{-1}(V);  \] hence $ f \in \beta_{\text{cts}}(\{b\},U_{1})   $ and $ g \in \beta_{\text{cts}}(\{b\},V_{1})  $. Now $ (f,g) \in \beta_{\text{cts}}(\{b\},U_{1}) \times \beta_{\text{cts}}(\{b\},V_{1}) $, which is an open subset of $ M_{\text{cts}} (B,A)_{\text{po}} \times  M_{\text{cts}} (B,A)_{\text{po}} $, and if \[ (f_{1},g_{1}) \in \beta_{\text{cts}}(\{b\},U_{1}) \times \beta_{\text{cts}}(\{b\},V_{1}),\] then $ (f_{1}(b),g_{1}(b)) \in U_{1} \times V_{1} \subseteq  m_{A}^{-1}(V)  $ so $ m_{\text{po}}(f_{1},g_{1}) \in \beta_{\text{cts}}(\{b\},V) $ hence, $ (f_{1},g_{1}) \in m_{\text{po}}^{-1}(\beta_{\text{cts}}(\{b\},V))  $. Therefore, \[ m_{\text{po}}^{-1}(\beta_{\text{cts}}(\{b\},V))  \subseteq M_{\text{cts}} (B,A)_{\text{po}} \times  M_{\text{cts}} (B,A)_{\text{po}} \] is an open subset, and hence $ m_{\text{po}} $ is a continuous map.
		
		Finally, to show the continuity of $  ac_{\text{po}} $, let $ (s,f) \in ac_{\text{po}} ^{-1}(\beta_{\text{cts}}(\{b\},V))  $. Then $ ac(s,f(b)) \in V $, i.e., $ (s,f(b)) \in ac^{-1}(V)  $, and $ ac^{-1}(V) $ is an open subset of $ G \times A $ since $ ac $ is continuous. Therefore, there exist open subsets $ U_{2} \subseteq G $ and $ V_{2} \subseteq A $ with $ (s , f(b)) \in U_{2} \times V_{2} \subseteq  ac^{-1}(V)  $. Now $ s \in U_{2}   $ and $ f \in \beta_{\text{cts}}(\{b\},V_{2})  $, hence $ (s,f) \in U_{2} \times \beta_{\text{cts}}(\{b\},V_{2}) $, which is an open subset of $ G\times M_{\text{cts}} (B,A)_{\text{po}} $,  and if $ (t,g) \in U_{2} \times \beta_{\text{cts}}(\{b\},V_{2})  $, then \[ (t,g(b)) \in U_{2} \times V_{2} \subseteq  ac_{\text{po}}^{-1}(\beta_{\text{cts}}(\{b\},V)),\] so $ ac_{\text{po}}(t,g) \in \beta_{\text{cts}}(\{b\},V) $ i.e.,  $ (t,g) \in ac_{\text{po}} ^{-1}(\beta_{\text{cts}}(\{b\},V))  $. This shows $ ac_{\text{po}}^{-1}(\beta_{\text{cts}}(\{b\},V)) $ is an open subset of $ G\times M_{\text{cts}} (B,A)_{\text{po}} $; therefore $ ac_{\text{po}} $ is a continuous map, and so the lemma follows.   
	\end{proof}
	\subsection{Cohomology Map}\label{SB4} In this subsection we define, for a $ G $-group $ A $, a continuous action of $ A $ on the right of $ Z^{1}_{\text{cts}}(G,A)_{\text{po}} $, and discuss some of the consequences that the continuity of this action entails.
	\begin{definition}[Cohomology Map]\label{D5}
		Let $ A $ be a $ G $-group. Define the \textit{cohomology map} of $ A $ by \[ cb=cb_{A} : Z^{1}_{\text{cts}}(G,A)_{\text{po}}\times A \longrightarrow M^{1}(G,A)_{\text{po}}, \ \ \text{given by} \ (a,x)\longmapsto cb(a,x)\stackrel{\mathrm{def}}{=} a.x \]
		such that $ a.x $ is given by  $ (a.x)_{s}\stackrel{\mathrm{def}}{=} x^{-1}a_{s} \ ^{s}x $, for all $ s \in G $. Two elements $ a,b \in Z^{1}_{\text{cts}}(G,A)_{\text{po}}  $ are said to be (A)-\textit{cohomologous} if there exists $ x \in A $ with $ a=b.x $. 
	\end{definition}
	\begin{proposition}\label{P1}
		Let $ A $ be a $ G $-group, and let $ cb $ be the map of Definition \ref{D5}. Then we have \[ cb : Z^{1}_{\text{cts}}(G,A)_{\text{po}}\times A \longrightarrow Z^{1}_{\text{cts}}(G,A)_{\text{po}}, \]
		and $ cb $ defines a continuous (for the product topology on $ Z^{1}_{\text{cts}}(G,A)_{\text{po}}\times A $) right action of $ A $ on $ Z^{1}_{\text{cts}}(G,A)_{\text{po}} $. 
	\end{proposition}
	\begin{proof}
		We first show the map $ cb $ is a well-defined action of $ A $ on the right of $ Z^{1}_{\text{cts}}(G,A)_{\text{po}} $. Recall we have \[ (a.x)_{s}\stackrel{\mathrm{def}}{=} x^{-1}a_{s} \ ^{s}x\stackrel{\mathrm{def}}{=} m_{A}(\iota_{A}(x),m_{A}(a_{s},ac(s,x)));\]
		therefore $ a.x $ can be given as the following composition of continuous maps 
		\[ 
		\begin{tikzcd}[row sep=1.1em , column sep=6em]
		G \arrow{r}[]{(\iota_{A}(x),a,ac_{\mid_{G\times\{x\}}})} & A \times A \times A \arrow{r}[]{m_{A}\times\text{id}_{A}} & A \times A \arrow{r}[]{m_{A}} & A. 
		\end{tikzcd} \]
		This shows $ a.x $ is a continuous map from $ G $ to $ A $. Now $ a.x $ is a cocycle since
		\[ (a.x)_{st} =x^{-1}a_{st} \ ^{st}x =x^{-1}a_{s} \ ^s a_{t} \ ^{s}(\ ^{t}x) = x^{-1}a_{s} \ ^{s}x \ ^{s}x^{-1} \ ^s a_{t} \ ^{s}(\ ^{t}x) \]
		\[=x^{-1}a_{s} \ ^{s}x \ ^{s} (x^{-1} a_{t} \ \ ^{t}x)= (a.x)_s \ ^{s}(a.x)_{t}, \ \text{for all} \ s,t \in G. \] 
		Above arguments show that the map $ cb $ is well-defined. To see $ cb $ defines a right action, note $ (a.1)_{s}=1^{-1}a_{s} \ ^{s}1=a_{s} $ for all $ s \in G $, so $ a.1=a $ for all $ a \in Z^{1}(G,A)_{\text{co}} $, also \[ (a.xy)_{s}= (xy)^{-1}a_{s} \ ^{s}(xy) = y^{-1} x^{-1} a_{s} \ ^{s}x \ ^{s}y= y^{-1} (a.x)_s\ ^{s}y = ((a.x).y)_{s} \]
		for all $ s \in G $ so $ a.xy= (a.x).y $, for all $ x,y \in A $ and all $ a \in Z^{1}(G,A)_{\text{co}} $. 
		
		Therefore, it remains to show that the map $ cb $ is continuous. To check the continuity of $ cb $ it suffices to show that the inverse image of any element of the subbase for the topology of $ Z^{1}_{\text{cts}}(G,A)_{\text{po}} $, via $ cb $, is an open subset of $  Z^{1}_{\text{cts}}(G,A)_{\text{po}}\times A $. Fix a point $ s \in G $ and an open subset $ V\subseteq A $, and let $ \beta_{\text{cts}}(\{s\},V)\stackrel{\mathrm{def}}{=}\{a\in Z_{\text{cts}}^{1}(G,A)_{\text{po}}: a_{s} \in V \} $ be an element of subbase for the point-open topology on $ Z^{1}_{\text{cts}}(G,A)_{\text{po}} $. We shall show every element of $ cb^{-1}(\beta_{\text{cts}}(\{s\},V)) $ has a neighbourhood contained in $ cb^{-1}(\beta_{\text{cts}}(\{s\},V)) $. Let $ (a,x) \in cb^{-1}(\beta_{\text{cts}}(\{s\},V)) $.  Then $ m_{A}(\iota_{A}(x),m_{A}(a_{s},ac(s,x))) \in V $, and so \[ (\iota_{A}(x),m_{A}(a_{s},ac(s,x))) \in m_{A}^{-1}(V). \] 
		
		Since $ m_{A} $ is a continuous map, $ m_{A}^{-1}(V) $ is an open subset of $ A\times A $; therefore there exist open subsets $ U_{1},V_{1} \subseteq A $ such that \[ (\iota_{A}(x),m_{A}(a_{s},ac(s,x))) \in U_{1}\times V_{1} \subseteq m_{A}^{-1}(V). \] 
		Thus, $ x \in \iota_{A}(U_{1}) $ and $ (a_{s},ac(s,x)) \in m_{A}^{-1}(V_{1}) $; once again, since $  m_{A} $ is continuous,  $ m_{A}^{-1}(V_{1}) $ is an open subset of $ A\times A $, so we can find open subsets $ U_{2},V_{2} \subseteq A $ such that \[ (a_{s},ac(s,x)) \in U_{2}\times V_{2} \subseteq m_{A}^{-1}(V_{1}). \] 
		Hence, $ a \in \beta_{\text{cts}}(\{s\},U_{2}) $ and $ (s,x) \in ac^{-1}(V_{2}) $, and since $ ac $ is a continuous map, there exist open subsets $ V_{3} \subseteq G $ and $ U_{3} \subseteq A $ with $ (s,x) \in V_{3} \times U_{3}\subseteq ac^{-1}(V_{2}) $.
		
		Now $ x \in U_{3}\cap \iota_{A}(U_{1}) $ and $ a \in \beta_{\text{cts}}(\{s\},U_{2}) $, so \[ (a,x) \in \beta_{\text{cts}}(\{s\},U_{2}) \times U_{3}\cap \iota_{A}(U_{1}), \] which is an open subset for the product topology on $ Z_{\text{cts}}^{1}(G,A)_{\text{po}}\times A $, and if \[ (b,y) \in \beta_{\text{cts}}(\{s\},U_{2}) \times U_{3}\cap \iota_{A}(U_{1}),\] then $ b(s) \in U_{2} $, $ y \in U_{3} $, and $ y^{-1} \in U_{1} $, so $ ac(s,y) \in V_{2} $, and therefore $ m_{A}(b(s),ac(s,y)) \in V_{1} $, hence $ m_{A}(y^{-1},m_{A}(b(s),ac(s,y)))=y^{-1}b(s) \ ^{s}y \in V $ which implies that $ (b.y)(s) \in V $ i.e., $ cb(b,y)= b.y \in \beta_{\text{cts}}(\{s\},V) $, and so $ (b,y) \in cb^{-1}(\beta_{\text{cts}}(\{s\},V)) $. This shows that $ cb^{-1}(\beta_{\text{cts}}(\{s\},V)) $ is an open subset of $ Z_{\text{cts}}^{1}(G,A)_{\text{po}}\times A $, which proves that $ cb $ is a continuous map and establishes the proposition. 
	\end{proof} 
	\begin{corollary}\label{C1}
		Let $ A $ be a $ G $-group, and $ cb : Z^{1}_{\text{cts}}(G,A)_{\text{po}}\times A \longrightarrow Z^{1}_{\text{cts}}(G,A)_{\text{po}} $ be the continuous action of $ A $ on $ Z_{\text{cts}}^{1}(G,A)_{\text{po}} $ as in Proposition \ref{P1}. Fix an element $ a \in Z_{\text{cts}}^{1}(G,A)_{\text{po}} $, and denote by $ \text{Stab}_{A}(a) \subseteq A $ the subset of element of $ A $ fixing $ a $. Then the subset $ \text{Stab}_{A}(a) \subseteq A $ is a closed subgroup. In particular, if $ A $ is a compact $ G $-group, then $ \text{Stab}_{A}(a) \subseteq A $ is a nonempty, compact, and Hausdorff space. 
	\end{corollary}
	\begin{proof}
		The subset $ \text{Stab}_{A}(a) \subseteq A $ is a subgroup. We shall show that $ \text{Stab}_{A}(a) \subseteq A $ is a closed subset. Recall we have the continuous map $ cb : Z_{\text{cts}}^{1}(G,A)_{\text{po}}\times A \longrightarrow Z_{\text{cts}}^{1}(G,A)_{\text{po}} $ by Proposition \ref{P1}, so we obtain, by restriction of $ cb $ to the subset $ \{a\} \times A $, a continuous map  $ cb_{\mid_{\{a\} \times A }} : \{a\} \times A \longrightarrow Z_{\text{cts}}^{1}(G,A)_{\text{po}} $. 	
		Note $ Z_{\text{cts}}^{1}(G,A)_{\text{po}} $ is a Hausdorff space, which follows from Lemma \ref{L3} i); this implies that $ \{a\} \subseteq Z_{\text{cts}}^{1}(G,A)_{\text{po}} $ is a closed subset. Hence $ cb_{\mid_{\{a\} \times A }}^{-1}(\{a\}) \subseteq \{a\} \times A $ is a closed subset, but $ cb_{\mid_{\{a\} \times A }}^{-1}(\{a\}) \cong\text{Stab}_{A}(a) $ under the homeomorphism $ \{a\} \times A\cong A $, so $ \text{Stab}_{A}(a) \subseteq A $ is a closed subset. 
		
		Alternatively, one can directly prove every element $ y \in A \setminus \text{Stab}_{A}(a) $ has a neighbourhood contained in $ A \setminus \text{Stab}_{A}(a) $ as follows. If $ y \in A \setminus \text{Stab}_{A}(a) $, then $ a.y\neq a $, so there exists $ s\in G $ with $ y^{-1}a_{s} \ ^{s}y\neq a_{s} $, which implies $ (y^{-1}a_{s}, ^{s}y) \notin m_{A}^{-1}(\{a_{s}\}) $. The space $ A $ is Hausdorff, so $ \{a_{s}\} \subset A $ is a closed subset; therefore $  m_{A}^{-1}(\{a_{s}\}) \subseteq A\times A $ is a closed subset of $ A\times A $; thus, $ A\times A \setminus m_{A}^{-1}(\{a_{s}\}) $ is an open subset of $ A\times A $. In particular, $ A\times A \setminus m_{A}^{-1}(\{a_{s}\}) $ contains the element $ (y^{-1}a_{s}, ^{s}y) $, so there exist open subsets $ U,V \subset A $ with $ y^{-1}a_{s} \in U $ and $ ^{s}y \in V $, and $ U\times V \subseteq A\times A \setminus m_{A}^{-1}(\{a_{s}\}) $. Therefore, $ y \in a_{s}U^{-1} \cap \ ^{s^{-1}}V $. Now since $ A $ is a topological group $ a_{s}U^{-1} $ is an open subset of $ A $, and $ G $ acting continuously on $ A $ implies $ ^{s^{-1}}V $ is an open subset of $ A $, hence $ a_{s}U^{-1} \cap \ ^{s^{-1}}V $ is an open neighbourhood of $ y $. 
		
		Note we have $ U\times V \cap m_{A}^{-1}(\{a_{s}\}) = \emptyset  $, so $ a_{s} \notin m_{A}(U\times V)=UV $. Now given $ z \in a_{s}U^{-1} \cap \ ^{s^{-1}}V $, we have $ z=a_{s}u^{-1}= \ ^{s^{-1}}v $ for some $ u \in U $ and $ v \in V $, so $ (a.z)_{s}= z^{-1}a_{s} \ ^{s}z=ua_{s}^{-1}a_{s} \ ^{s}(^{s^{-1}}v)=uv \neq a_{s} $ i.e., $ (a.z)_{s}\neq a_{s} $, so $ z \notin \text{Stab}_{A}(a) $; thus $ a_{s}U^{-1} \cap \ ^{s^{-1}}V \subset A \setminus \text{Stab}_{A}(a) $, which shows $ A\setminus\text{Stab}_{A}(a) $ is an open subset of $ A $; therefore $ \text{Stab}_{A}(a) $ is a closed subset of $ A $. 
		
		Now a closed subset of a compact and Hausdorff space is a compact and Hausdorff space cf. \cite[p.~119, 17.14 Theorem. a)]{MR2048350} \& cf. \cite[p.~87, 13.8 Theorem. a)]{MR2048350}; so if $ A $ is a compact $ G $-group, then $ \text{Stab}_{A}(a) \subseteq A $ is a  nonempty, compact, and Hausdorff space.
	\end{proof}
	\begin{corollary}\label{C2}
		Let $ A $ be a $ G $-group. Then the subset $ A^{G} \subseteq A $, the fixed subgroup of $ A $ under the action of $ G $, is a closed.
	\end{corollary}
	\begin{proof}
		By Corollary \ref{C1}, the stabiliser of each element of $ Z_{\text{cts}}^{1}(G,A)_{\text{co}} $ is a closed subgroup of $ A $, so to prove $ A^{G} \subseteq A $ is a closed subset, it suffices to show $ A^{G}=\text{Stab}_{A}(a_{1}) $. But this is immediate, since let $ x \in A^{G} $. Then $ ^{s}x=x $ for all $ s \in G $ i.e., $ x^{-1} \ ^{s}x=1 $ for all $ s \in G $ so $ x^{-1} a_{1,s} \ ^{s}x=a_{1,s} $  for all $ s \in G $ i.e., $ a_{1}.x=a_{1} $ so $ x \in \text{Stab}_{A}(a_{1}) $. Conversely, if $ x \in \text{Stab}_{A}(a_{1}) $, then $ x^{-1} a_{1,s} \ ^{s}x=a_{1,s} $  for all $ s \in G $ so $ ^{s}x=x $ for all $ s \in G $, hence $ x \in A^{G} $. 
	\end{proof}
	\begin{corollary}\label{C3}
		Let $ A $ be a compact $ G $-group. Fix $ a \in Z_{\text{cts}}^{1}(G,A)_{\text{po}} $, and let $ \text{Orb}_{A}(a) $ be the orbit of the element $ a $ under the action of $ A $. Then the subset $ \text{Orb}_{A}(a) \subseteq Z_{\text{cts}}^{1}(G,A)_{\text{po}} $, is a nonempty, compact, and Hausdorff subspace. 
	\end{corollary}
	\begin{proof}
		By Proposition \ref{P1}, the map $ cb: Z_{\text{cts}}^{1}(G,A)_{\text{po}}\times A \longrightarrow Z_{\text{cts}}^{1}(G,A)_{\text{po}} $ is a continuous action. Therefore, we obtain a continuous map \[ cb_{\mid_{\{a\} \times A }} : \{a\} \times A \longrightarrow Z_{\text{cts}}^{1}(G,A)_{\text{po}}, \]	
		by restriction of $ cb $ to subspace $ \{a\} \times A \subseteq Z^{1}(G,A)_{\text{po}}\times A $. In particular, the image of $ cb_{\mid_{\{a\} \times A }}  $ is $ \text{Orb}_{A}(a) $. Therefore, we have a surjective continuous map
		\[ cb_{\mid_{\{a\} \times A }} : \{a\} \times A \longrightarrow \text{Orb}_{A}(a), \]
		and $ cb_{\mid_{\{a\} \times A }}(a,x)=cb_{\mid_{\{a\} \times A }}(a,y) $ if and only if $ xy^{-1} \in \text{Stab}_{A}(a) $. Hence we have a commutative diagram 
		\[ 
		\begin{tikzcd}[row sep=1.5em , column sep=3em]
		\{a\} \times A \arrow[two heads]{d}[]{}\arrow{r}[]{cb_{\mid_{\{a\} \times A }}} & \text{Orb}_{A}(a) \arrow[equal]{d}[]{} \\ \{a\} \times A/\text{Stab}_{A}(a) \arrow{r}[]{\overline{cb}_{\mid_{\{a\} \times A }}} & \text{Orb}_{A}(a).
		\end{tikzcd} \]
		
		Note $ A $ is a compact and Hausdorff space, so by Corollary \ref{C1}, $ \text{Stab}_{A}(a) $ is a compact and Hausdorff topological group. Now $ \text{Stab}_{A}(a)  $ acts continuously on the right of $ A $, then $ A/\text{Stab}_{A}(a) $ is a compact space being continuous image of the compact space $ A $ under the continuous open map $ A \twoheadrightarrow A/\text{Stab}_{A}(a) $ cf. \cite[p.~119, Theorem 17.1]{MR2048350}, and $ A/\text{Stab}_{A}(a) $ is Hausdorff cf. \cite[p.~38, Theorem 3.1, (1)]{MR0413144}. Since $ Z_{\text{cts}}^{1}(G,A)_{\text{po}} $ is a Hausdorff space, $ \text{Orb}_{A}(a) \subseteq Z_{\text{cts}}^{1}(G,A)_{\text{po}} $ is a Hausdorff space cf. \cite[p.~87, 13.8 Theorem. a)]{MR2048350}. Now the map \[ \overline{cb}_{\mid_{\{a\} \times A }} :\{a\} \times A/\text{Stab}_{A}(a)\longrightarrow \text{Orb}_{G}(a) \] is a continuous bijection from a compact space to a Hausdorff space, so $ \overline{cb}_{\mid_{\{a\} \times A }} $ is a homeomorphism cf. \cite[p.~123, 17.14 Theorem.]{MR2048350}; therefore \[ A/\text{Stab}_{A}(a)\cong \{a\} \times A/\text{Stab}_{A}(a) \cong \text{Orb}_{A}(a). \] 
		
		The above shows $ \text{Orb}_{A}(a) $ is a compact and Hausdorff space, it contains $ a $, so it is a nonempty, compact, and Hausdorff space. 
	\end{proof}
	\begin{remark}\label{R3}
		Recall a $ G $-group homomorphism $ \varphi : A\longrightarrow A' $ between $ G $-groups $ A $ and $ A' $ induces a continuous map $ \varphi_{\ast} : Z^{1}_{\text{cts}}(G,A)_{\text{po}}\longrightarrow Z^{1}_{\text{cts}}(G,A')_{\text{po}} $ by Remark \ref{R1}. Now for $ (a,x) \in  Z^{1}(G,A)_{\text{po}}\times A $ we have $ \varphi_{\ast} cb_{A}(a,x)=cb_{A'}(\varphi_{\ast}(a),\varphi(x)) $ i.e., $ \varphi $ induces the following commutative diagram \[ 
		\begin{tikzcd}[row sep=1.5em , column sep=2.0em]
		Z^{1}_{\text{cts}}(G,A)_{\text{po}}\times A \arrow[]{d}[]{cb_{A}}\arrow{r}[]{\varphi_{\ast} \times \varphi } & Z^{1}_{\text{cts}}(G,A')_{\text{po}}\times A' \arrow{d}[]{cb_{A'}} \\ Z^{1}_{\text{cts}}(G,A)_{\text{po}} \arrow{r}[]{\varphi_{\ast}} & Z^{1}_{\text{cts}}(G,A')_{\text{po}}.
		\end{tikzcd} \]
	\end{remark}	
	\subsection{First Nonabelian Continuous Cohomology Pointed Spaces}\label{SB5} 
	In this subsection, for a $ G $-group $ A $, we define the first nonabelian continuous cohomology pointed space of $ G $ with coefficients in $ A $ as the quotient space $ Z_{\text{cts}}^{1}(G,A)_{\text{po}} $ by the continuous action of $ A $ given by $ cb $, and we investigate the properties this space. But first we note that the zeroth cohomology group of $ G $ with coefficients in a $ G $-group $ A $ is defined as $  H_{\text{cts}}^{0}(G,A)_{\text{po}}=H_{\text{cts}}^{0}(G,A)\stackrel{\mathrm{def}}{=} A^{G} $, which is a closed subgroup of $ A $ by Corollary \ref{C2}.     
	\begin{definition}[First Nonabelian Continuous Cohomology Pointed Spaces]\label{D6}
		\footnote{Compare with \cite[p.~45, 5.1]{MR1867431}.} Let $ A $ be a $ G $-group. Then we define the \textit{first nonabelian continuous cohomology pointed space of $ G $ with coefficients in $ A $} as the quotient space $ Z_{\text{cts}}^{1}(G,A)_{\text{po}} $ by the action of $ A $ i.e., \[ H_{\text{cts}}^{1}(G,A)_{\text{po}}\stackrel{\mathrm{def}}{=} Z_{\text{cts}}^{1}(G,A)_{\text{po}}/A. \] 
	\end{definition}
	
	In other words, we have $ H_{\text{cts}}^{1}(G,A)_{\text{po}}=Z_{\text{cts}}^{1}(G,A)_{\text{po}}/\sim $,  where for $ a,b \in Z_{\text{cts}}^{1}(G,A)_{\text{po}} $ we write $ a\sim b $ if $ a $ and $ b $ are cohomologous. The set $ H_{\text{cts}}^{1}(G,A)_{\text{po}} $ is endowed with the quotient topology; in particular, the natural surjective map \[\pi_{A} : Z_{\text{cts}}^{1}(G,A)_{\text{po}}\longrightarrow H_{\text{cts}}^{1}(G,A)_{\text{po}}, \] which sends an element $ a \in Z_{\text{cts}}^{1}(G,A)_{\text{po}} $ to its orbit under the action of $ A $ is a continuous open map cf. \cite[p.~37]{MR0413144}. If we denote the image of an element $ a \in Z_{\text{cts}}^{1}(G,A)_{\text{po}} $ under $ \pi_{A} $ by $ [a] $, then the class of trivial cocycle $ [a_{1}] $ gives a distinguished point of $ H_{\text{cts}}^{1}(G,A)_{\text{po}} $.
	\begin{remark}\label{R4}
		Recall that by Remark \ref{R3} a $ G $-group homomorphism $ \varphi : A\longrightarrow A' $ between $ G $-groups $ A $ and $ A' $ induces a commutative \[ 
		\begin{tikzcd}[row sep=1.5em , column sep=2.0em]
		Z^{1}_{\text{cts}}(G,A)_{\text{po}}\times A \arrow[]{d}[]{cb_{A}}\arrow{r}[]{\varphi_{\ast} \times \varphi } & Z^{1}_{\text{cts}}(G,A')_{\text{po}}\times A' \arrow{d}[]{cb_{A'}} \\ Z^{1}_{\text{cts}}(G,A)_{\text{po}} \arrow{r}[]{\varphi_{\ast}} & Z^{1}_{\text{cts}}(G,A')_{\text{po}}.
		\end{tikzcd} \]
		Therefore, $ \varphi $ induces a commutative diagram
		\[ 
		\begin{tikzcd}[row sep=1.5em , column sep=2.0em]
		Z^{1}_{\text{cts}}(G,A)_{\text{po}} \arrow[two heads]{d}[]{\pi_{A}}\arrow{r}[]{\varphi_{\ast} } & Z^{1}_{\text{cts}}(G,A')_{\text{po}} \arrow[two heads]{d}[]{\pi_{A'}} \\ H^{1}_{\text{cts}}(G,A)_{\text{po}} \arrow{r}[]{\overline{\varphi}_{\ast}} & H^{1}_{\text{cts}}(G,A')_{\text{po}},
		\end{tikzcd} \]
		where for $ [a] \in H^{1}_{\text{cts}}(G,A)_{\text{po}} $ we set $ \overline{\varphi}_{\ast}([a])\stackrel{\mathrm{def}}{=} [\varphi_{\ast}(a)] $. The map $ \overline{\varphi}_{\ast} $ is continuous as follows. Let $ V \subseteq H^{1}_{\text{cts}}(G,A')_{\text{po}}  $ be an open subset. Then by the commutativity of the diagram above we have $ \overline{\varphi}_{\ast}\pi_{A}=\pi_{A'}\varphi_{\ast} $, so \[ (\overline{\varphi}_{\ast}\pi_{A})^{-1}(V)=\pi_{A}^{-1}(\overline{\varphi}_{\ast}^{-1}(V))=(\pi_{A'}\varphi_{\ast})^{-1}(V)=\varphi_{\ast}^{-1}(\pi_{A'}^{-1}(V)), \]
		since $ \pi_{A} $ is surjective, applying  $ \pi_{A} $ to the both side of equation above we find
		\[ \pi_{A}(\pi_{A}^{-1}(\overline{\varphi}_{\ast}^{-1}(V)))=\overline{\varphi}_{\ast}^{-1}(V) =\pi_{A}(\varphi_{\ast}^{-1}(\pi_{A'}^{-1}(V)) ).\]
		Now since $ \pi_{A'} $ and $ \varphi_{\ast} $ are continuous maps and $ \pi_{A} $ is an open map we conclude that $ \overline{\varphi}_{\ast}^{-1}(V) $ is an open subset of $ H^{1}_{\text{cts}}(G,A)_{\text{po}} $.
		
		Similarly, if $ \psi : H\longrightarrow G $ is a continuous group homomorphism, by Remark \ref{R1} and calculations similar to the one above, one obtains a commutative diagram of continuous maps
		\[ 
		\begin{tikzcd}[row sep=1.5em , column sep=2.0em]
		Z^{1}_{\text{cts}}(G,A)_{\text{po}} \arrow[two heads]{d}[]{\pi_{A}}\arrow{r}[]{\psi^{\ast}} & Z^{1}_{\text{cts}}(H,A)_{\text{po}} \arrow[two heads]{d}[]{\pi_{A}} \\ H^{1}_{\text{cts}}(G,A)_{\text{po}} \arrow{r}[]{\overline{\psi}^{\ast}} & H^{1}_{\text{cts}}(H,A)_{\text{po}},
		\end{tikzcd} \] 
		where for $ [a] \in H^{1}_{\text{cts}}(G,A)_{\text{po}} $ we set $ \overline{\psi}^{\ast}([a])\stackrel{\mathrm{def}}{=} [\psi^{\ast}(a)] $. 	  
	\end{remark}
	
	Next remark discusses the well-known correspondence between cohomology classes of continuous cocycles and isomorphism classes of torsors.
	\begin{remark}\label{R5}
		Let $ \text{Tors}_{\text{cts}}^{G}(A) $ be the set of continuous $ G $-torsors under $ A $ as in Definition \ref{D4}. Then by Remark \ref{R2} we have a homeomorphism 
		\[\mathcal{T}: Z^{1}_{\text{cts}}(G,A)_{\text{po}}\longrightarrow \text{Tors}^{G}_{\text{cts}}(A)_{\text{po}}. \]
		
		Now suppose for $ a,b \in Z^{1}_{\text{cts}}(G,A)_{\text{po}}  $, we have $ a=b.x $ for some $ x \in A $. Then we obtain a homeomorphism of $ G $,$A$-spaces \[ _{b}P\longrightarrow \ _{a}P \ \ \text{given by} \ \  p\longmapsto x^{-1}p\stackrel{\mathrm{def}}{=} m_{A}(x^{-1},p). \] 
		
		Conversely, given two $ G $-torsors $ P_{1}, P_{2} \in \text{Tors}^{G}_{\text{cts}}(A)_{\text{po}} $ and a homeomorphism of $ G $,$A$ -spaces $ \alpha :  P_{1} \longrightarrow  P_{2} $. Then by Remark \ref{R2} we can find canonical homeomorphisms of $ G $,$A$ -spaces  $ P_{1}\cong  \ _{a}P  $ and $ P_{2}\cong \  _{b}P  $ for some $ a,b \in Z^{1}_{\text{cts}}(G,A)_{\text{po}} $. Therefore, $ \alpha $ induces a homeomorphism of $ G $,$A$-spaces (also denoted by $ \alpha $), $ \alpha: \ _{a}P \longrightarrow \  _{b}P $. In particular, since $ \alpha(p.x)=\alpha(p).x $, we see that $ \alpha(p)=\alpha(1p)=\alpha(1)p $ i.e., $ \alpha $ is just multiplication on the left on $ A $ by $ \alpha(1) $. Further, we must have $ \alpha(\ ^{s_{'}}p)=\ ^{s_{'}}\alpha(p) $, so 
		\[ \alpha(\ ^{s_{'}}p)=\alpha(a_{s} \ ^{s}p)=\alpha(1)a_{s} \ ^{s}p=\ ^{s_{'}}\alpha(p)=b_{s}\ ^{s}\alpha(p)=b_{s}\ ^{s}\alpha(1) \ ^{s}p, \] 
		Therefore, $ a_{s} =\alpha(1)^{-1}b_{s}\ ^{s}\alpha(1) $ for all $ s \in G $ i.e., $ a=b.\alpha(1) $. Let $ \text{Tors}^{G}_{\text{cts}}(A)_{\text{po}}/\text{Isom} $ be the set of $ G $,$A$-isomorphism classes of $ G $-torsors under $ A $ endowed with quotient topology, and denote by $ [P] $ the $ G $,$A$-isomorphism classes of a $ G $-torsors $ P $. Then the above arguments shows that $ \mathcal{T} $ descends to a homeomorphism \[  H^{1}_{\text{cts}}(G,A)_{\text{po}}\longrightarrow  \text{Tors}^{G}_{\text{cts}}(A)_{\text{po}}/\text{Isom} \ \ \text{given by} \ \ [a]\longmapsto [\ _{a}P ]. \]      
	\end{remark}
	\begin{lemma}\label{L9}
		Suppose $ A $ is a compact $ G $-group. Then
		\begin{enumerate}[i)]
			\item The topological space $ H_{\text{cts}}^{1}(G,A)_{\text{po}} $ is Hausdorff.
			\item The topological space $ Z_{\text{cts}}^{1}(G,A)_{\text{po}} $ is compact if and only if the topological space $ H_{\text{cts}}^{1}(G,A)_{\text{po}} $ is compact; in particular if $ G $ is compact and Hasudorff, and $ A $ is evenly continuous with respect to $ G $, then $ H_{\text{cts}}^{1}(G,A)_{\text{po}} $ is a compact space.	
		\end{enumerate} 
	\end{lemma}
	\begin{proof}
		By Proposition \ref{P1}, the cohomology map $ cb : Z_{\text{cts}}^{1}(G,A)_{\text{po}}\times A \longrightarrow Z_{\text{cts}}^{1}(G,A)_{\text{po}} $ is continuous; by Lemma \ref{L3}, $ Z_{\text{cts}}^{1}(G,A)_{\text{po}} $ is a Hausdorff space. Now $ A $ is a compact and Hausdorff topological group which acts continuously on a Hausdorff space $ Z_{\text{cts}}^{1}(G,A)_{\text{po}} $. Therefore $ Z_{\text{cts}}^{1}(G,A)_{\text{po}} $ is an $ A $-space with $ A $ a compact space, which matches the definition of \cite[p.~32, 1]{MR0413144}. Now i) follows from  \cite[p.~38, Theorem 3.1, 1)]{MR0413144}, and ii) follows from \cite[p.~38, Theorem 3.1, 4)]{MR0413144}; in particular, if $ G $ is compact and Hausdorff space, and $ A $ is evenly continuous space with respect to $ G $, by Lemma \ref{L7}, $ Z_{\text{cts}}^{1}(G,A)_{\text{po}} $ is a compact space, which implies that the space $ H_{\text{cts}}^{1}(G,A)_{\text{po}} $ is a compact space cf. \cite[p.~119, Theorem 17.1]{MR2048350}.
	\end{proof} 
	\section{The First Main Theorem}\label{S3} 
	In this section we shall work with compact $ G $-groups. Recall $ R $ is a fixed directed poset and we assume $ \{A_{r}, \varphi_{rs}, R\} $ is an inverse system of compact topological $ G $-groups, where the transition maps are $ G $-group homomorphism. Then the first main result of this article is the following.
	\begin{theorem}\label{T1}
		Let $ A $ be a compact $ G $-group. Assume $ A $ has a presentation $ A= \varprojlim_{\stackrel{ \ }{r\in R}}A_{r} $. Denote by $ \varphi_{r}: A\longrightarrow A_{r} $ the natural projections. Then there exists a continuous bijection \[ \Theta : H_{\text{cts}}^{1}(G,A)_{\text{po}} \longrightarrow \varprojlim_{\stackrel{ \ }{r\in R}} H_{\text{cts}}^{1}(G,A_{r})_{\text{po}} \ \ \text{given by} \ \ [a]\longmapsto ([\varphi_{r\ast}(a)]). \]
		In particular, if $ G $ is a compact and Hausdorff space, and $ A $ is evenly continuous with respect to $ G $, then $ \Theta $ is a homeomorphism.     	
	\end{theorem}
	We prove Theorem \ref{T1}, using 3 lemmas, by showing there exists a well-defined natural continuous map $ \Theta $, induced by $ \varphi_{r} $ for $ r \in R $, such that $ \Theta $ is a continuous bijection.  
	\begin{lemma}\label{L10}
		Let $ A $ be a compact $ G $-group satisfying the assumption of Theorem \ref{T1}. Then there exists a well-defined continuous map \[ \Theta : H^{1}_{\text{cts}}(G,A)_{\text{po}} \longrightarrow \varprojlim_{\stackrel{ \ }{r\in R}} H^{1}_{\text{cts}}(G,A_{r})_{\text{po}}  \ \ \text{given by} \ \ [a]\longmapsto ([\varphi_{r\ast}(a)]) \] and we have a commutative diagram
		\[ 
		\begin{tikzcd}[row sep=1.5em , column sep=2.0em]
		Z_{\text{cts}}^{1}(G,A)_{\text{po}} \arrow[two heads]{d}[]{\pi_{A}}\arrow{r}[]{\theta} & \varprojlim_{\stackrel{ \ }{r\in R}} Z_{\text{cts}}^{1}(G,A_{r})_{\text{po}} \arrow{d}[]{\varprojlim_{\stackrel{ \ }{r\in R}}\pi_{A_{r}}} \\ H^{1}_{\text{cts}}(G,A)_{\text{po}} \arrow{r}[]{\Theta} & \varprojlim_{\stackrel{ \ }{r\in R}} H^{1}_{\text{cts}}(G,A_{r})_{\text{po}}.
		\end{tikzcd} \]
	\end{lemma}
	\begin{proof}
		Since $ \varphi_{rt}: A_{r}\longrightarrow A_{t} $ for $ r\geq t $ are continuous $ G $-group homomorphisms, by Remark \ref{R4}, they induce a commutative diagram of continuous maps
		\[ 
		\begin{tikzcd}[row sep=1.5em , column sep=2.0em]
		Z_{\text{cts}}^{1}(G,A_{r})_{\text{po}} \arrow[two heads]{d}[]{\pi_{A_{r}}}\arrow{r}[]{\varphi_{rt\ast}} & Z_{\text{cts}}^{1}(G,A_{t})_{\text{po}} \arrow[two heads]{d}[]{\pi_{A_{t}}} \\ H^{1}_{\text{cts}}(G,A_{r})_{\text{po}} \arrow{r}[]{\overline{\varphi}_{rt\ast}} & H^{1}_{\text{cts}}(G,A_{t})_{\text{po}}.
		\end{tikzcd} \] 
		Note $ \varphi_{rt} $ are transition maps, and so $ \{H^{1}_{\text{cts}}(G,A_{r})_{\text{po}},\overline{\varphi}_{rt\ast}, R \} $ is inverse system of topological pointed sets. 
		
		Now the continuous maps $ \varphi_{r}: A\longrightarrow A_{r} $ for $ r \in R $ are $ G $-group homomorphisms, so by Remark \ref{R4} they induces a commutative diagram of continuous maps
		\[ 
		\begin{tikzcd}[row sep=1.5em , column sep=2.0em]
		Z_{\text{cts}}^{1}(G,A)_{\text{po}} \arrow[two heads]{d}[]{\pi_{A}}\arrow{r}[]{\varphi_{r\ast}} & Z_{\text{cts}}^{1}(G,A_{r})_{\text{po}} \arrow[two heads]{d}[]{\pi_{A_{r}}} \\ H^{1}_{\text{cts}}(G,A)_{\text{po}} \arrow{r}[]{\overline{\varphi}_{r\ast}} & H^{1}_{\text{cts}}(G,A_{r})_{\text{po}},
		\end{tikzcd} \]
		compatible with transition maps of $ \{Z^{1}_{\text{cts}}(G,A_{r})_{\text{po}},\varphi_{rt\ast}, R \} $ and $ \{H^{1}_{\text{cts}}(G,A_{r})_{\text{po}},\overline{\varphi}_{rt\ast}, R \} $. Taking the inverse limit of the above diagram we obtain a unique commutative diagram 
		\[ 
		\begin{tikzcd}[row sep=1.5em , column sep=2.0em]
		Z_{\text{cts}}^{1}(G,A)_{\text{po}} \arrow[two heads]{d}[]{\pi_{A}}\arrow{r}[]{\theta} & \varprojlim_{\stackrel{ \ }{r\in R}} Z_{\text{cts}}^{1}(G,A_{r})_{\text{po}} \arrow{d}[]{\varprojlim_{\stackrel{ \ }{r\in R}}\pi_{A_{r}}} \\ H^{1}_{\text{cts}}(G,A)_{\text{po}} \arrow{r}[]{\Theta} & \varprojlim_{\stackrel{ \ }{r\in R}} H^{1}_{\text{cts}}(G,A_{r})_{\text{po}}.
		\end{tikzcd} \] 
		This shows that the map $ \Theta $ is well-defined and is given by $ [a]\longmapsto ([\varphi_{r\ast}(a)]) $ for $ [a] \in H^{1}_{\text{cts}}(G,A)_{\text{po}} $. Now it follows from Lemma \ref{L4} that the map $ \theta $ is continuous (homeomorphism in fact since $ A $ is compact), and $ \varprojlim_{\stackrel{ \ }{r\in R}}\pi_{A_{r}} $ is a continuous map (being inverse limit of continuous maps), also since $ \pi_{A} $ is an open map, by preforming a similar calculation to what is done is Remark \ref{R4}, we find that $ \Theta $ is a continuous map. 
	\end{proof} 
	\begin{lemma}\label{L11}
		Let $ A $ be a compact $ G $-group satisfying the assumption of Theorem \ref{T1}. Then the continuous map \[ \Theta : H^{1}_{\text{cts}}(G,A)_{\text{po}} \longrightarrow \varprojlim_{\stackrel{ \ }{r\in R}} H^{1}_{\text{cts}}(G,A_{r})_{\text{po}}, \] 
		defined in Lemma \ref{L10}, is injective.
	\end{lemma}
	\begin{proof}
		Let $ [a], [b] \in H^{1}_{\text{cts}}(G,A)_{\text{po}} $, and suppose \[ \Theta([a])=([\varphi_{r}a])=\Theta([b])=([\varphi_{r}b]). \] Then for all $ r \in R $ there exists $ x_{r} \in A_{r} $ such that $ \varphi_{r}a=\varphi_{r}b.x_{r} $. Therefore, the set $ S_{r}=\{\widetilde{x} \in A_{r}: \varphi_{r}a=\varphi_{r}b.\widetilde{x} \} $ is nonempty for all $ r \in R $. Now $ S_{r} $ is equal to the coset $ \text{Stab}_{A_{r}}(\varphi_{r}b)x_{r} $, where $ \text{Stab}_{A_{r}}(\varphi_{r}b) $ is the stabiliser of $ \varphi_{r}b $ for the action of $ A_{r} $ on $ Z_{\text{cts}}^{1}(G,A_{r})_{\text{po}} $. The set $ \text{Stab}_{A_{r}}(\varphi_{r}b) \subseteq A_{r} $ is a closed, compact, and Hausdorff subgroup by Corollary \ref{C1}. Therefore, $ S_{r}=\text{Stab}_{A_{r}}(\varphi_{r}b) x_{r} $ is nonempty, compact, and Hausdorff space for each $ r \in R $. In particular, if $ r\geq t $ and $ \widetilde{x} \in S_{r} $, then \[ \varphi_{t}b .\varphi_{rt}(\widetilde{x})=\varphi_{rt}\varphi_{r}b.\varphi_{rt}(\widetilde{x})=\varphi_{rt}(\varphi_{r}b.\widetilde{x})=\varphi_{rt}\varphi_{r}a=\varphi_{t}a, \]
		so $ \varphi_{rt}(\widetilde{x}) \in S_{t} $. Hence, the set $ \{S_{r}, \varphi_{rt}, R\} $ is an inverse system of nonempty, compact, and Hausdorff spaces. Therefore, $ \varprojlim_{\stackrel{ \ }{r \in R}} S_{r} $ is nonempty, compact, and Hausdorff cf. \cite[p.~4, Proposition 1.1.4]{MR2599132}. Now for $ x \in \varprojlim_{\stackrel{ \ }{r \in R}} S_{r} \hookrightarrow \varprojlim_{\stackrel{ \ }{r \in R}} A_{r} =A $, one has \[ \varphi_{r} b.x=\varphi_{r}b.\varphi_{r}(x)=\varphi_{r}a \] for all $ r \in R $, so $ \theta(b.x)=\theta(a) $, and since by Lemma \ref{L4} the map $ \theta $ is injective, we conclude that $ b.x=a $; therefore $ [a]=[b] $, and this proves that the map $ \Theta $ is injective. 
	\end{proof}
	\begin{lemma}\label{L12}
		Let $ A $ be a compact $ G $-group satisfying the assumption of Theorem \ref{T1}. Then the continuous map \[ \Theta : H^{1}_{\text{cts}}(G,A)_{\text{po}} \longrightarrow \varprojlim_{\stackrel{ \ }{r\in R}} H^{1}_{\text{cts}}(G,A_{r})_{\text{po}}, \] 
		defined in Lemma \ref{L10}, is surjective.
	\end{lemma}
	\begin{proof}
		Recall we have surjective continuous maps \[\pi_{A_{r}}: Z^{1}_{\text{cts}}(G,A_{r})_{\text{po}} \longrightarrow H^{1}_{\text{cts}}(G,A_{r})_{\text{po}}  \] 
		for each $ r \in R $, which are compatible with the transition maps of $ \{Z^{1}_{\text{cts}}(G,A_{r})_{\text{po}},\varphi_{rt\ast}, R \} $ and $ \{H^{1}_{\text{cts}}(G,A_{r})_{\text{po}},\overline{\varphi}_{rt\ast}, R \} $. Let $ ([a_{r}]) \in \varprojlim_{\stackrel{ \ }{r\in R}} H^{1}_{\text{cts}}(G,A_{r})_{\text{po}} $. Then since $ A_{r} $ is compact and Hausdorff for each $ r \in R $, the set $ \pi_{A_{r}}^{-1}([a_{r}])=\text{Orb}_{A_{r}}(a_{r}) $ is a nonempty, compact, and Hausdorff space for each $ r \in R $ by Corollary \ref{C3}. In particular the set $ \{\text{Orb}_{A_{r}}(a_{r}), \varphi_{rt\ast}, R\} $ is an inverse system of nonempty, compact, and Hausdorff spaces. Therefore, the set $ \varprojlim_{\stackrel{ \ }{r \in R}} \text{Orb}_{A_{r}}(a_{r}) $ is nonempty, compact, and Hausdorff cf. \cite[p.~4, Proposition 1.1.4]{MR2599132}. Now for \[ a \in \varprojlim_{\stackrel{ \ }{r \in R}} \text{Orb}_{A_{r}}(a_{r}) \hookrightarrow \varprojlim_{\stackrel{ \ }{r \in R}} Z^{1}_{\text{cts}}(G,A_{r})_{\text{po}} \] we have $ \pi_{A}\theta^{-1}(a) \in H^{1}_{\text{cts}}(G,A)_{\text{po}} $ and $ \Theta\pi_{A}\theta^{-1}(a)=([a_{r}]) $; this proves that the map $ \Theta $ is surjective. 
	\end{proof}
	
	Now we provide the proof of the Theorem \ref{T1}. We recall the statement of Theorem \ref{T1} again.
	\begin{theorem}[]
		Let $ A $ be a compact $ G $-group. Assume $ A $ has a presentation $ A= \varprojlim_{\stackrel{ \ }{r\in R}}A_{r} $. Denote by $ \varphi_{r}: A\longrightarrow A_{r} $ the natural projections. Then there exists a continuous bijection \[ \Theta : H_{\text{cts}}^{1}(G,A)_{\text{po}} \longrightarrow \varprojlim_{\stackrel{ \ }{r\in R}} H_{\text{cts}}^{1}(G,A_{r})_{\text{po}} \ \ \text{given by} \ \ [a]\longmapsto ([\varphi_{r\ast}(a)]). \]
		In particular, if $ G $ is a compact and Hausdorff space, and $ A $ is evenly continuous with respect to $ G $, then $ \Theta $ is a homeomorphism.   
	\end{theorem}
	\begin{proof}
		By Lemma \ref{L10}, there exists a commutative diagram of continuous maps
		\[ 
		\begin{tikzcd}[row sep=1.5em , column sep=2.0em]
		Z_{cts}^{1}(G,A)_{\text{co}} \arrow[two heads]{d}[]{\pi_{A}}\arrow{r}[]{\theta} & \varprojlim_{\stackrel{ \ }{r\in R}} Z_{cts}^{1}(G,A_{r})_{\text{co}} \arrow{d}[]{\varprojlim_{\stackrel{ \ }{r\in R}}\pi_{A_{r}}} \\ H^{1}_{cts}(G,A) \arrow{r}[]{\Theta} & \varprojlim_{\stackrel{ \ }{r\in R}} H^{1}_{cts}(G,A_{r}),
		\end{tikzcd} \]
		by Lemma \ref{L11} the map $ \Theta $ is injective, and by Lemma \ref{L12} the map $ \Theta $ is surjective. Therefore, $ \Theta $ is a continuous bijection. 
		
		In particular, by Lemma \ref{L9}, i), the space $ H^{1}_{cts}(G,A_{r}) $ is Hausdorff for all $ r \in R $; hence $ \varprojlim_{\stackrel{ \ }{r\in R}} H^{1}_{cts}(G,A_{r}) $ is a Hausdorff space cf. \cite[p.~87, 13.8 Theorem. a) \& b)]{MR2048350}; now if $ G $ is a compact and Hausdorff space, and $ A $ is evenly continuous with respect to $ G $, then By Lemma \ref{L9}, ii), the space $ H^{1}_{cts}(G,A) $ is a compact space. Therefore, in such case $ \Theta $ is a continuous bijection from a compact space to a Hausdorff space, which implies that $ \Theta $ is a homeomorphism cf. \cite[p.~123, 17.14 Theorem]{MR2048350}. 
	\end{proof}
	\subsection{Applications of The First Main Theorem}\label{SB6}  
	As an application of Theorem \ref{T1}, one can relate the continuous cohomology sets with coefficients in a finitely generated profinite group, to the inverse limit of cohomology sets with coefficients in finite discrete groups. Another application of Theorem \ref{T1} is concerned with the cohomology set with coefficients in a prosolvable group. We first discuss a general lemma which applies to both of these applications of Theorem \ref{T1}. Recall a profinite groups is a compact, Hausdorff and totally disconnected topological group, which is also an evenly continuous space by Corollary \ref{C3}. 
	\begin{proposition}\label{P2}
		Let $ A $ be a profinite $ G $-group, and assume a collection $ \{N_{r}\}_{r \in R} $ of characteristic closed subgroups of $ A $ is given such that $ \cap_{r \in R}N_{r}=1 $ and $ N_{r} \subseteq N_{t} $ whenever $ r\geqslant t $. Let $ A_{r}=A/N_{r} $. Then $ A $ has a presentation $ A\cong \varprojlim_{\stackrel{ \ }{r\in R}}A_{r} $ by profinite $ G $-groups $ A_{r} $, and $ \varprojlim_{\stackrel{ \ }{r\in R}}A_{r} $ satisfies the assumption of Theorem \ref{T1}; in particular there exists a continuous bijection \[ \Theta : H_{\text{cts}}^{1}(G,A)_{\text{po}} \longrightarrow \varprojlim_{\stackrel{ \ }{r\in R}} H_{\text{cts}}^{1}(G,A_{r})_{\text{po}} \ \ \text{given by} \ \ [a]\longmapsto ([\varphi_{r,\ast}(a)]); \]
		further if $ G $ is a compact and Hausdorff space, and $ A $ is evenly continuous with respect to $ G $, then $ \Theta $ is a homeomorphism.   
	\end{proposition}
	\begin{proof}
		Since $ N_{r} $ is a closed normal subgroup of $ A $ for all $ r\in R $, the quotient group $ A_{r} $ is a profinite group for all $ r \in R $ cf. \cite[p.~28, Proposition 2.2.1, (a)]{MR2599132}. By Lemma \ref{L6}, the group $ G $ acts by automorphism on $ A $; in particular, since $ N_{r} $ is a characteristic subgroup of $ A $, the exact sequence
		\[ 
		\begin{tikzcd}[row sep=1.1em , column sep=1.2em]
		1 \arrow[]{r}[]{} &  N_{r} \arrow{r}[]{}  & A \arrow{r}[]{}   &  A_{r} \arrow{r}[]{} & 1,  
		\end{tikzcd} \] 
		is an exact sequence of $ G $-groups. Therefore, the set $ \{A_{r}\}_{r \in R} $ is an inverse system, where the obvious transition $ A_{r}\twoheadrightarrow A_{t} $, whenever $ r\geq t $, are $ G $-group homomorphisms. Now since $ \varprojlim_{\stackrel{ \ }{r\in R}} $ is an exact functor on the category of profinite groups cf. \cite[p.~31, Proposition 2.2.4]{MR2599132}, taking the inverse limit we obtain an exact sequence 
		\[ 
		\begin{tikzcd}[row sep=1.1em , column sep=1.2em, ]
		1 \arrow[]{r}[]{} & \varprojlim_{\stackrel{ \ }{r \in R}} N_{r} \arrow{r}[]{}  & \varprojlim_{\stackrel{ \ }{r \in R}}A \arrow{r}[]{}   & \varprojlim_{\stackrel{ \ }{r \in R}} A_{r} \arrow{r}[]{} & 1,  
		\end{tikzcd} \]
		where $ \varprojlim_{\stackrel{ \ }{r \in R}} N_{r}\cong\cap_{r \in R}N_{r}=1 $, and $ \varprojlim_{\stackrel{ \ }{r \in R}}A \cong A $, so we have a natural homeomorphism  $ A\cong \varprojlim_{\stackrel{ \ }{r\in R}}A_{r} $. 
		
		Now $ \varprojlim_{\stackrel{ \ }{r \in R}}A_{r} $ satisfies the assumption of Theorem \ref{T1}; therefore, we have a continuous bijection  \[ \Theta : H_{\text{cts}}^{1}(G,\varprojlim_{\stackrel{ \ }{r\in R}}A_{r})_{\text{po}} \longrightarrow \varprojlim_{\stackrel{ \ }{r\in R}} H_{\text{cts}}^{1}(G,A_{r})_{\text{po}}. \] 
		We further have a homeomorphism \[  H_{\text{cts}}^{1}(G,A)_{\text{po}} \cong H_{\text{cts}}^{1}(G,\varprojlim_{\stackrel{ \ }{r \in R}}A_{r})_{\text{po}} \] 
		induced by the homeomorphism $ A\cong \varprojlim_{\stackrel{ \ }{r\in R}}A_{r} $. Hence, the map \[ H_{\text{cts}}^{1}(G,A)_{\text{po}} \longrightarrow \varprojlim_{\stackrel{ \ }{r\in R}} H_{\text{cts}}^{1}(G,A_{r})_{\text{po}}, \] which is the composition \[ H_{\text{cts}}^{1}(G,A)_{\text{po}}\cong H_{\text{cts}}^{1}(G,\varprojlim_{\stackrel{ \ }{r \in R}}A_{r})_{\text{po}}\longrightarrow \varprojlim_{\stackrel{ \ }{r\in R}} H_{\text{cts}}^{1}(G,A_{r})_{\text{po}} \] is a continuous bijection. In particular, according to Theorem \ref{T1}, if $ G $ is a compact and Hausdorff space, and $ A $ is evenly continuous with respect to $ G $, then $ \Theta $ is a homeomorphism.    
	\end{proof}  
	\begin{corollary}\label{C4}
		Suppose $ A $ is a finitely generated profinite $ G $-group. Then $ A $ has a presentation $ A\cong \varprojlim_{\stackrel{ \ }{r\in \mathbb{N}}}A_{r} $ by finite $ G $-groups $ A_{r} $ for $ r\in \mathbb{N} $ satisfying the assumption of Theorem \ref{T1}; so there exists a there exist a continuous bijection \[ \Theta : H_{\text{cts}}^{1}(G,A)_{\text{po}} \longrightarrow \varprojlim_{\stackrel{ \ }{r\in \mathbb{N}}} H_{\text{cts}}^{1}(G,A_{r})_{\text{po}}, \] 
		and if $ G $ is a compact and Hausdorff space, and $ A $ is evenly continuous with respect to $ G $, then $ \Theta $ is a homeomorphism.   
	\end{corollary}
	\begin{proof}
		Since $ A $ is a finitely generated profinite group, $ 1 \in A $ has a fundamental system of neighbourhoods consisting of a countable chain of open characteristic subgroups \[A=N_{0}\supseteq  N_{1} \supseteq  N_{2}, ...  \] 
		cf. \cite[p.~44, Proposition 2.5.1]{MR2599132}. Therefore, we have a natural homeomorphism $ A\cong\varprojlim_{\stackrel{ \ }{r\in \mathbb{N}}}A_{r} $, and $ \varprojlim_{\stackrel{ \ }{r\in \mathbb{N}}}A_{r} $ satisfies the assumption of Theorem \ref{T1}. Now the corollary follows from Proposition \ref{P2}.		
	\end{proof}
	\begin{corollary}\label{C5}
		Suppose $ A $ is a prosolvable\footnote{See \cite[p.~19]{MR2599132} for definition.} $ G $-group. Then $ A $ has a presentation $ A\cong \varprojlim_{\stackrel{ \ }{i\geq 1}}A_{i} $ by profinite $ G $-groups $ A_{i} $, for $ i \in \mathbb{N} $, satisfying the assumption of Theorem \ref{T1}; so there exists a continuous bijection \[ H^{1}_{\text{cts}}(G,A)_{\text{po}} \longrightarrow \varprojlim_{\stackrel{ \ }{i\geq 1}} H^{1}_{\text{cts}}(G,A_{i})_{\text{po}}, \] 
		and if $ G $ is a compact and Hausdorff space, and $ A $ is evenly continuous with respect to $ G $, then $ \Theta $ is a homeomorphism.
	\end{corollary}
	\begin{proof}
		Suppose $ A $ is a prosolvable group. Then one can write $ A =\varprojlim_{\stackrel{ \ }{r \in R}}A_{r} $ with $ A_{r} $ finite solvable group for each $ r \in R $, and continuous surjective homomorphisms $ \varphi_{r} : A \longrightarrow A_{r} $ such that $ \cap_{r\in R}\Ker\varphi_{r}=1 $ cf. \cite[p.~22, Theorem 2.1.3, (c)]{MR2599132}. 
		
		Define $ A(i+1)\stackrel{\mathrm{def}}{=}\overline{[A(i):A(i)]} $ to be the closure of the commutator subgroup of $ A(i) $, i.e., $ [A(i):A(i)] $ is generated by commutator elements of $ A(i) $, for $ i \in \mathbb{N} $, and let $ A(0)=A $. By induction on $ i $, one can show that $ A(i) $ is a characteristic subgroup of $ A $ for every $ i \geq 1 $; we set $ A_{i}=A/A(i) $. Now $ \varphi_{r}(A(i)) \subseteq A_{r}(i) $ for all $ i \geq 1 $, and since for all $ r \in R $ there exists $ d_{r}\in \mathbb{N} $ with $ A_{r}(d_{r}) =1 $ (since $ A_{r} $ is solvable for each $ r \in R $)  one has $ \varphi_{r}(A(d_{r}))=1 $, hence $ \cap_{i\geq1}A(i) \subset A(d_{r}) \subset \Ker\varphi_{r} $ for all $ r \in R $ i.e., $ \cap_{i\geq 1}A(i) \subseteq \cap_{r} \Ker\varphi_{r} =1 $. Therefore, one has a natural homeomorphism $ A \cong\varprojlim_{\stackrel{ \ }{i \geq 1}}A_{i} $, and $ \varprojlim_{\stackrel{ \ }{i \geq 1}}A_{i} $ satisfies the assumption of Theorem \ref{T1}. Now the corollary follows from Proposition \ref{P2}. 	
	\end{proof}
	\section{$ G $-Modules and The Second Main Theorem}\label{S4} 
	In this section we shall fix $ G $ to be a compact and Hausdorff topological group, and work with abelian $ G $-groups, which also known, and we shall refer to, as $ G $-modules, also $ R $ is a fixed directed poset and we assume $ \{A_{r}, \varphi_{rs}, R\} $ is an inverse system of compact topological $ G $-modules, where the transition maps are $ G $-group homomorphism. Given a $ G $-module $ A $, one can define the continuous cohomology sets $ H^{n}_{\text{cts}}(G,A) $ for all $ n\geq0 $ cf. \cite[p.~137]{MR2392026}. Denote by $ G^{n} $ the product of $ n $ copies of $ G $ endowed with product topology. Then the main result of this article is the following.
	\begin{theorem}\label{T2}
		Let $ A $ a compact $ G $-module and fix $ n\geq 1 $. Assume $ A $ has a presentation $ A= \varprojlim_{\stackrel{ \ }{r\in R}}A_{r} $, where $ A_{r} $ is evenly continuous with respect to $ G^{n-1} $ for all $ r\in R $. Denote by $ \varphi_{r}: A\longrightarrow A_{r} $ the natural projections. Then there exists a continuous bijection \[ \Theta_{n} : H_{\text{cts}}^{n}(G,A)_{\text{po}} \longrightarrow \varprojlim_{\stackrel{ \ }{r\in R}} H_{\text{cts}}^{n}(G,A_{r})_{\text{po}} \ \ \text{given by} \ \ [a]\longmapsto ([\varphi_{r\ast}(a)]). \]
		In particular, if $ A $ is evenly continuous with respect to $ G^{n} $, then $ \Theta_{n} $ is a homeomoprhism.
	\end{theorem}
	
	The proof of Theorem \ref{T2} is somewhat similar to the proof of Theorem \ref{T1}. We fix $ n\geq1 $ and start by introducing the point-open topology on the set of $ n $-cocycles of $ G $ with values in $ A $,  $ Z_{\text{cts}}^{n}(G,A) $. Since $ A_{r} $ are compact $ G $-modules, which are evenly continuous with respect to $ G^{n-1} $ and for all $ r \in R $, we have that $ M_{\text{cts}}(G^{n-1},A) $ is a compact and Hausdorff space. Then we interpret $ H_{\text{cts}}^{n}(G,A) $ as a quotient space of $ Z_{\text{cts}}^{n}(G,A) $ by a continuous action of a compact and Hausdorff topological group $ M_{\text{cts}}(G^{n-1},A) $, and from these we deduce our theorem.
	
	We fix a $ G $-module $ A $, sometimes we shall use ''$ + $'' instead of the operation of $ m_{A} $ of $ A $.  As before, we denote by $ M_{\text{cts}}(G^{n},A)_{\text{po}} $ the set of all continuous maps from $ G^n\longrightarrow A $ endowed with point-open topology. Recall we let $ M_{\text{cts}}(G^{0},A)_{\text{po}} $ be the set of all constant maps from $ G\longrightarrow A $, so we have natural identification $ M_{\text{cts}}(G^{0},A)_{\text{po}}\cong A $. It follows from Lemma \ref{L8} that $ M_{\text{cts}}(G^{n},A)_{\text{po}} $ is a $ G $-module for all $ n\geq 1 $. Now there exists a well-known \textit{differential homomorphism} \[d_{n}:M_{\text{cts}}(G^{n-1},A)_{\text{po}}\longrightarrow M_{\text{cts}}(G^{n},A)_{\text{po}} \ \ \text{given by} \ \ f\longmapsto d_{n}f  \] for $ n\geq1 $, defined by
	\[d_{n}f_{s_{1},...,s_{n}}\stackrel{\mathrm{def}}{=} \ ^{s_{1}}f_{s_{2},...,s_{n}}+ \sum_{i=1}^{n-1}(-1)^{i}f_{s_{1},...,s_{i}s_{i+1},...s_{n}}+(-1)^{n}f_{s_{1},..,s_{n-1}}, \]      
	which has the property that $ d_{n+1}d_{n}=0 $. 
	\begin{lemma}\label{L13}
		Let $ A $ be a $ G $-module. Then for each $ n\geq1 $ the map  \[d_{n}:M_{\text{cts}}(G^{n-1},A)_{\text{po}}\longrightarrow M_{\text{cts}}(G^{n},A)_{\text{po}} \ \ \text{given by} \ \ f\longmapsto d_{n}f,  \]
		defined above, is a continuous group homomorphism.
	\end{lemma}
	\begin{proof}
		Fix $ n\geq1 $. Then the map $ d_{n} $ is a group homomorphism, so we only require to that show $ d_{n} $ is a continuous map. Fix a point $ \textbf{s}\stackrel{\mathrm{def}}{=}(s_{1},...,s_{n}) \in G^{n}  $ and an open set $ V \subseteq A $. Let $ \beta_{\text{cts}}(\{\textbf{s}\},V) \subseteq M_{\text{cts}}(G^{n},A)_{\text{po}}  $ be an element of the subbase for the point-open topology on $ M_{\text{cts}}(G^{n},A)_{\text{po}}  $. Let $ f \in d_{n}^{-1}(\beta_{\text{cts}}(\{\textbf{s}\},V)) $. Then \[d_{n}f_{s_{1},...,s_{n}}\stackrel{\mathrm{def}}{=} m_{A}\left( \ ^{s_{1}}f_{s_{2},...,s_{n}}, m_{A}\left(\sum_{i=1}^{n-1}(-1)^{i}f_{s_{1},...,s_{i}s_{i+1},...s_{n}},(-1)^{n}f_{s_{1},..,s_{n-1}}\right)\right) \in V, \]
		since $ m_{A} $ is a continuous map and $ V $ is an open subset, there exists open subsets $ U_{1},V_{1} \subseteq A $ such that \[\left( \ ^{s_{1}}f_{s_{2},...,s_{n}}, m_{A}\left(\sum_{i=1}^{n-1}(-1)^{i}f_{s_{1},...,s_{i}s_{i+1},...s_{n}},(-1)^{n}f_{s_{1},..,s_{n-1}}\right)\right) \in U_{1}\times V_{1} \subseteq m_{A}^{-1}(V). \] 
		Therefore, $ ^{s_{1}}f_{s_{2},...,s_{n}} \in U_{1} $ and $ m_{A}\left(\sum_{i=1}^{n-1}(-1)^{i}f_{s_{1},...,s_{i}s_{i+1},...s_{n}},(-1)^{n}f_{s_{1},..,s_{n-1}}\right) \in V_{1}  $. Now since the action map $ ac: G\times A \longmapsto A $ is continuous, there exists open subsets $ U_{2} \subseteq G $ and $ V_{2} \subseteq A $ such that \[(s_{1},f_{s_{2},...,s_{n}}) \in U_{2}\times V_{2} \subseteq ac^{-1}(U_{1}),   \]
		so $ s_{1} \in U_{2} $ and $ f \in \beta_{\text{cts}}(\{(s_{2},...,s_{n})\},V_{2}) $. Now since \[ m_{A}\left(\sum_{i=1}^{n-1}(-1)^{i}f_{s_{1},...,s_{i}s_{i+1},...s_{n}},(-1)^{n}f_{s_{1},..,s_{n-1}}\right) \in V_{1}, \] by continuity of $ m_{A} $, there exists open subsets $ U_{3},V_{3} \subseteq A $ such that \[\left(\sum_{i=1}^{n-1}(-1)^{i}f_{s_{1},...,s_{i}s_{i+1},...s_{n}},(-1)^{n}f_{s_{1},..,s_{n-1}}\right) \in V_{3}\times U_{3} \subseteq m_{A}^{-1}(V_{1}), \] 
		so $ f \in \beta_{\text{cts}}(\{(s_{1},...,s_{n-1})\},(-1)^{n}U_{3})  $ and \[\sum_{i=1}^{n-1}(-1)^{i}f_{s_{1},...,s_{i}s_{i+1},...s_{n}}\stackrel{\mathrm{def}}{=} m_{A}\left(-f_{s_{1}s_{2},...,s_{n}}, \sum_{i=2}^{n-1}(-1)^{i}f_{s_{1},...,s_{i}s_{i+1},...s_{n}}\right) \in V_{3}. \]
		Again by continuity of $ m_{A} $, there exists open subsets $ U_{4},V_{4} \subseteq A $ such that \[\left(-f_{s_{1}s_{2},...,s_{n}}, \sum_{i=2}^{n-1}(-1)^{i}f_{s_{1},...,s_{i}s_{i+1},...s_{n}}\right) \in U_{4}\times V_{4} \subseteq m_{A}^{-1}(V_{3}), \]
		so $ f \in \beta_{\text{cts}}(\{(s_{1}s_{2},...,s_{n})\},-U_{4})  $ and \[\sum_{i=2}^{n-1}(-1)^{i}f_{s_{1},...,s_{i}s_{i+1},...s_{n}}\stackrel{\mathrm{def}}{=} m_{A}\left(f_{s_{1},s_{2}s_{3},...,s_{n}}, \sum_{i=3}^{n-1}(-1)^{i}f_{s_{1},...,s_{i}s_{i+1},...s_{n}}\right) \in V_{4}. \] 
		And again by the continuity of $ m_{A} $, there exists open subsets $ U_{5},V_{5} \subseteq A $ such that \[\left(f_{s_{1},s_{2}s_{3},...,s_{n}}, \sum_{i=3}^{n-1}(-1)^{i}f_{s_{1},...,s_{i}s_{i+1},...s_{n}}\right) \in U_{5}\times V_{5} \subseteq m_{A}^{-1}(V_{4}), \]
		so $ f \in \beta_{\text{cts}}(\{(s_{1},s_{2}s_{3},...,s_{n})\}, U_{5})  $ and \[\sum_{i=3}^{n-1}(-1)^{i}f_{s_{1},...,s_{i}s_{i+1},...s_{n}}\stackrel{\mathrm{def}}{=} m_{A}\left(f_{s_{1},s_{2},s_{3}s_{4},...,s_{n}}, \sum_{i=4}^{n-1}(-1)^{i}f_{s_{1},...,s_{i}s_{i+1},...s_{n}}\right) \in V_{5}. \] 
		Therefore, continuing in this form, and if we denote by $ \textbf{s}_{j}\stackrel{\mathrm{def}}{=} (s_{1},...,s_{j}s_{j+1},...,s_{n}) \in G^{n-1} $, we find that open subset $ U_{j+3},V_{j+3} \subseteq A $ such that
		\[\left(f_{s_{1},...,s_{j}s_{j+1},...,s_{n}}, \sum_{i=j}^{n-1}(-1)^{i}f_{s_{1},...,s_{i}s_{i+1},...s_{n}}\right) \in U_{j+3}\times V_{j+3} \subseteq m_{A}^{-1}(V_{j+2}), \]
		so $ f_{\textbf{s}_{j}} \in (-1)^{j}U_{3+j} $ and \[\sum_{i=j}^{n-1}(-1)^{i}f_{s_{1},...,s_{i}s_{i+1},...s_{n}}\stackrel{\mathrm{def}}{=} m_{A}\left(f_{s_{1},...,s_{j}s_{j+1},...,s_{n}}, \sum_{i=j}^{n-1}(-1)^{i}f_{s_{1},...,s_{i}s_{i+1},...s_{n}}\right) \in V_{j+3}, \]
		for $ j=1,...,n-2 $; therefore, we find 
		\begin{scriptsize}
			\[f \in \left(\cap_{j=1}^{n-2}\beta_{\text{cts}}(\{\textbf{s}_{j}\}, (-1)^{j}U_{3+j})\right) \cap \beta_{\text{cts}}(\{\textbf{s}_{n-1}\}, (-1)^{n-1}V_{n+1}) \cap \beta_{\text{cts}}(\{(s_{2},...,s_{n})\},V_{2}) \cap \beta_{\text{cts}}(\{(s_{1},...,s_{n-1})\},(-1)^{n}U_{3}),  \]          
		\end{scriptsize}
		which is an open subset of $ M_{\text{cts}}(G^{n-1},A)_{\text{po}}  $, and if we take 
		\begin{scriptsize}
			\[g \in \left(\cap_{j=1}^{n-2}\beta_{\text{cts}}(\{\textbf{s}_{j}\}, (-1)^{j}U_{3+j})\right) \cap \beta_{\text{cts}}(\{\textbf{s}_{n-1}\}, (-1)^{n-1}V_{n+1}) \cap \beta_{\text{cts}}(\{(s_{2},...,s_{n})\},V_{2}) \cap \beta_{\text{cts}}(\{(s_{1},...,s_{n-1})\},(-1)^{n}U_{3}),  \]          
		\end{scriptsize}
		then one checks that $ g \in d_{n}^{-1}(\beta_{\text{cts}}(\{\textbf{s}\},V)) $, which show that $ d_{n}^{-1}(\beta_{\text{cts}}(\{\textbf{s}\},V)) $ is an open subset of $ M_{\text{cts}}(G^{n-1},A)_{\text{po}} $ i.e., $ d_{n} $ is a continuous map.     
	\end{proof}
	
	Now since $ A $ is a Hausdorff space, by Lemma \ref{L8} $ M_{\text{cts}}(G^{n},A)_{\text{po}} $ is a Hausdorff topological group for all $ n\geq0 $, so the subset containing the zero map $ \{0\} \subseteq M_{\text{cts}}(G^{n+1},A)_{\text{po}} $ is a closed subset. We define \[Z^{n}_{\text{cts}}(G,A)_{\text{po}} \stackrel{\mathrm{def}}{=} d_{n+1}^{-1}(\{0\})\stackrel{\mathrm{def}}{=} \Ker d_{n+1},\] by the continuity of the homomorphism $ d_{n+1} $, as proved in Lemma \ref{L17}, the subset $ Z^{n}_{\text{cts}}(G,A)_{\text{po}} \subseteq M_{\text{cts}}(G^{n},A)_{\text{po}} $ is closed. In particular, if $ A $ is a compact and evenly continuous with respect to $ G^{n} $, then we deduce, using Lemma \ref{L5}, that both $ Z^{n}_{\text{cts}}(G,A)_{\text{po}} $ and $ M_{\text{cts}}(G^{n},A)_{\text{po}} $ are compact spaces for all $ n\geq 0 $.    
	\begin{proposition}\label{P3}
		Let $ A $ be a $ G $-module. Then for each $ n\geq1 $ we have a map \[ cb_{n} : Z^{n}_{\text{cts}}(G,A)_{\text{po}}\times M_{\text{cts}}(G^{n-1},A)_{\text{po}} \longrightarrow Z^{n}_{\text{cts}}(G,A)_{\text{po}} \ \ \text{given by} \ \  (a,f)\longmapsto a+d_{n}f,  \]
		which defines a continuous (for the product topology on $ Z^{n}_{\text{cts}}(G,A)_{\text{po}}\times M_{\text{cts}}(G^{n-1},A)_{\text{po}} $) right action of $ M_{\text{cts}}(G^{n-1},A)_{\text{po}} $ on $ Z^{n}_{\text{cts}}(G,A)_{\text{po}} $. 
	\end{proposition}
	\begin{proof}
		Fix $ n\geq1 $. Note first that since $ d_{n+1}d_{n}=0 $ and $ Z^{n}_{\text{cts}}(G,A)_{\text{po}} \stackrel{\mathrm{def}}{=} d_{n+1}^{-1}(\{0\}) $, the map $ cb_{n} $ is well-defined. Now It follows from Lemma \ref{L5}, i), and the definition of $ Z^{n}_{\text{cts}}(G,A)_{\text{po}} $, that both $ Z^{n}_{\text{cts}}(G,A)_{\text{po}} $ and $ M_{\text{cts}}(G^{n-1},A)_{\text{po}} $ are Hausdorff topological groups. Therefore, the group operation of $ Z^{n}_{\text{cts}}(G,A)_{\text{po}} $ is a continuous map; since by Lemma \ref{L13} $ d_{n} $ is a continuous map, we conclude that $ cb_{n} $ is a continuous map; finally, it is clear that $ cb_{n} $ is an action of $ M_{\text{cts}}(G^{n-1},A)_{\text{po}} $ on $ Z^{n}_{\text{cts}}(G,A)_{\text{po}} $.    
	\end{proof}
	\begin{corollary}\label{C6}
		Let $ A $ be a $ G $-module, fix $ n\geq 1 $ and an element $ a \in Z_{\text{cts}}^{n}(G,A)_{\text{po}} $, and denote by $ \text{Stab}_{n,A}(a)  $ the subset of elements of $ M_{\text{cts}}(G^{n-1},A)_{\text{po}} $ fixing $ a $. Then the subset $ \text{Stab}_{n,A}(a) \subseteq M_{\text{cts}}(G^{n-1},A)_{\text{po}}  $ is a closed subgroup. In particular, if $ A $ is a compact $ G $-module, which is evenly continuous with respect to $ G^{n-1} $, then $ \text{Stab}_{n,A}(a) \subseteq M_{\text{cts}}(G^{n-1},A)_{\text{po}} $ is a nonempty, compact, and Huasdorff space. 
	\end{corollary}
	\begin{proof}
		Note we have $ \text{Stab}_{n,A}(a)=\Ker d_{n}\stackrel{\mathrm{def}}{=} Z_{\text{cts}}^{n-1}(G,A)_{\text{po}} $ which is a closed subgroup of $ M_{\text{cts}}(G^{n-1},A)_{\text{po}} $. Now if $ A $ is a compact $ G $-module, which is evenly continuous with respect to $ G^{n-1} $, then by Lemma \ref{L5} the space $ M_{\text{ecst}}(G^{n-1},A)_{\text{po}} $ is compact, and since $ Z_{\text{cts}}^{n-1}(G,A)_{\text{po}} $ is a closed subset of a compact space, we conclude that it is also compact cf. \cite[p.~119, 17.5 Theorem. a)]{MR2048350}, now the proof of the corollary follows. 
	\end{proof}
	\begin{corollary}\label{C7}
		Fix $ n\geq 1 $. Let $ A $ be a compact $ G $-module, which is evenly continuous with respect to $ G^{n-1} $, and let $ a \in Z_{\text{cts}}^{n}(G,A)_{\text{po}} $. Let $ \text{Orb}_{n,A}(a) $ be the orbit of the element $ a $ under the action of $ M_{\text{cts}}(G^{n-1},A)_{\text{po}} $. If $ A $ , then the subset $ \text{Orb}_{n,A}(a) \subseteq Z_{\text{cts}}^{n}(G,A)_{\text{po}} $, is a nonempty, compact, and Hausdorff subspace. 
	\end{corollary}
	\begin{proof}
		By Proposition \ref{P3}, the map $ cb_{n}: Z_{\text{cts}}^{n}(G,A)_{\text{po}}\times M_{\text{cts}}(G^{n-1},A)_{\text{po}} \longrightarrow Z_{\text{cts}}^{n}(G,A)_{\text{po}} $ is continuous. Therefore, we obtain a continuous map \[ cb_{\mid_{\{a\} \times M_{\text{cts}}(G^{n-1},A)_{\text{po}} }} : \{a\} \times M_{\text{cts}}(G^{n-1},A)_{\text{po}} \longrightarrow Z_{\text{cts}}^{n}(G,A)_{\text{po}}, \]	
		by restriction of $ cb_{n} $ to subspace $ \{a\} \times M_{\text{cts}}(G^{n-1},A)_{\text{po}} \subseteq Z^{n}(G,A)_{\text{po}}\times M_{\text{cts}}(G^{n-1},A)_{\text{po}} $. In particular, the image of $ cb_{\mid_{\{a\} \times A }}  $ is $ \text{Orb}_{n,A}(a) $. Therefore, we have a surjective continuous map
		\[ cb_{\mid_{\{a\} \times M_{\text{cts}}(G^{n-1},A)_{\text{po}} }} : \{a\} \times M_{\text{cts}}(G^{n-1},A)_{\text{po}} \longrightarrow \text{Orb}_{n,A}(a), \]
		and $ cb_{\mid_{\{a\} \times M_{\text{cts}}(G^{n-1},A)_{\text{po}} }}(a,x)=cb_{\mid_{\{a\} \times M_{\text{cts}}(G^{n-1},A)_{\text{po}} }}(a,y) $ if and only if $ xy^{-1} \in \text{Stab}_{n,A}(a) $. Hence we have a commutative diagram 
		\[ 
		\begin{tikzcd}[row sep=1.5em , column sep=7em]
		\{a\} \times M_{\text{cts}}(G^{n-1},A)_{\text{po}} \arrow[two heads]{d}[]{}\arrow{r}[]{cb_{\mid_{\{a\} \times M_{\text{cts}}(G^{n-1},A)_{\text{po}} }}} & \text{Orb}_{n,A}(a) \arrow[equal]{d}[]{} \\ \{a\} \times M_{\text{cts}}(G^{n-1},A)_{\text{po}}/\text{Stab}_{n,A}(a) \arrow{r}[]{\overline{cb}_{\mid_{\{a\} \times M_{\text{cts}}(G^{n-1},A)_{\text{po}} }}} & \text{Orb}_{n,A}(a).
		\end{tikzcd} \]
		
		By Corollary \ref{C6}, $ \text{Stab}_{n,A}(a) $ is a compact and Hausdorff topological group. Now $ \text{Stab}_{n,A}(a)  $ acts continuously on the right of $ M_{\text{cts}}(G^{n-1},A)_{\text{po}} $, then $ M_{\text{cts}}(G^{n-1},A)_{\text{po}}/\text{Stab}_{n,A}(a) $ is a compact space being continuous image of the compact space $ M_{\text{cts}}(G^{n-1},A)_{\text{po}} $ under the continuous open map $ M_{\text{cts}}(G^{n-1},A)_{\text{po}} \twoheadrightarrow M_{\text{cts}}(G^{n-1},A)_{\text{po}}/\text{Stab}_{n,A}(a) $ cf. \cite[p.~119, Theorem 17.1]{MR2048350}, and $ M_{\text{cts}}(G^{n-1},A)_{\text{po}}/\text{Stab}_{n,A}(a) $ is Hausdorff cf. \cite[p.~38, Theorem 3.1, (1)]{MR0413144}. Since $ Z_{\text{cts}}^{n}(G,A)_{\text{po}} $ is a Hausdorff space, $ \text{Orb}_{n,A}(a) \subseteq Z_{\text{cts}}^{n}(G,A)_{\text{po}} $ is a Hausdorff space cf. \cite[p.~87, 13.8 Theorem. a)]{MR2048350}. Now \[ \overline{cb}_{\mid_{\{a\} \times M_{\text{cts}}(G^{n-1},A)_{\text{po}} }} :\{a\} \times M_{\text{cts}}(G^{n-1},A)_{\text{po}}/\text{Stab}_{n,A}(a)\longrightarrow \text{Orb}_{n,A}(a) \] is a continuous bijection from a compact space to a Hausdorff space, so $ \overline{cb}_{\mid_{\{a\} \times M_{\text{cts}}(G^{n-1},A)_{\text{po}} }} $ is a homeomorphism cf. \cite[p.~123, 17.14 Theorem]{MR2048350}; therefore \[ M_{\text{cts}}(G^{n-1},A)_{\text{po}}/\text{Stab}_{n,A}(a)\cong \{a\} \times M_{\text{cts}}(G^{n-1},A)_{\text{po}}/\text{Stab}_{n,A}(a) \cong \text{Orb}_{n,A}(a). \] 
		
		The above shows $ \text{Orb}_{n,A}(a) $ is a compact and Hausdorff space, it contains $ a $, so it is a nonempty, compact, and Hausdorff space. 
	\end{proof}
	\begin{remark}\label{R6}
		A $ G $-group homomorphism $ \varphi : A\longrightarrow A' $ between $ G $-groups $ A $ and $ A' $ induces a continuous map, for each $ n\geq1 $, $ \varphi_{\ast} : Z^{n}_{\text{cts}}(G,A)_{\text{po}}\longrightarrow Z^{n}_{\text{cts}}(G,A')_{\text{po}} $, and for $ (a,f) \in  Z^{n}(G,A)_{\text{po}}\times M_{\text{cts}}(G^{n-1},A)_{\text{po}} $ we have $ \varphi_{\ast} cb_{n,A}(a,f)=cb_{n,A'}(\varphi_{\ast}(a),\varphi_{\ast}(f)) $ i.e., $ \varphi $ induces the following commutative diagram \[ 
		\begin{tikzcd}[row sep=1.5em , column sep=2em]
		Z^{n}_{\text{cts}}(G,A)_{\text{po}}\times M_{\text{cts}}(G^{n-1},A)_{\text{po}} 	\arrow[]{d}[]{cb_{n,A}}\arrow{r}[]{\varphi_{\ast} \times \varphi_{\ast} } & Z^{n}_{\text{cts}}(G,A')_{\text{po}}\times M_{\text{cts}}(G^{n-1},A')_{\text{po}} \arrow{d}[]{cb_{n,A'}} \\ Z^{n}_{\text{cts}}(G,A)_{\text{po}} \arrow{r}[]{\varphi_{\ast}} & Z^{n}_{\text{cts}}(G,A')_{\text{po}},
		\end{tikzcd} \]
		which follows from Lemma \ref{L1} and properties of $ \varphi $.
	\end{remark}
	
	Now the $ n $-th continuous cohomology group of $ G $ with coefficients in $ A $ can be defined as the quotient space of $ Z_{\text{cts}}^{n}(G,A)_{\text{po}} $, by the action of $ M_{\text{cts}}(G^{n-1},A)_{\text{po}} $ i.e., \[ H_{\text{cts}}^{n}(G,A)_{\text{po}}\stackrel{\mathrm{def}}{=} Z_{\text{cts}}^{n}(G,A)_{\text{po}}/M_{\text{cts}}(G^{n-1},A)_{\text{po}}, \ \ \text{for} \  n\geq1 . \] 
	For each $ n\geq1 $, the set $ H_{\text{cts}}^{n}(G,A)_{\text{po}} $ is endowed with the quotient topology; in particular, the natural surjective homomorphism \[\pi_{n,A} : Z_{\text{cts}}^{n}(G,A)_{\text{po}}\longrightarrow H_{\text{cts}}^{n}(G,A)_{\text{po}}, \] which sends an element $ a \in Z_{\text{cts}}^{n}(G,A)_{\text{po}} $ to its orbit under the action of $ M_{\text{cts}}(G^{n-1},A)_{\text{po}}  $ is a continuous open map cf. \cite[p.~37]{MR0413144}, also $ H_{\text{cts}}^{n}(G,A)_{\text{po}} $ is a topological group. 
	\begin{remark}\label{R7}
		Fix $ n\geq1 $. Then by Remark \ref{R6} a $ G $-group homomorphism $ \varphi : A\longrightarrow A' $ between $ G $-groups $ A $ and $ A' $ induces a commutative \[ 
		\begin{tikzcd}[row sep=1.5em , column sep=2em]
		Z^{n}_{\text{cts}}(G,A)_{\text{po}}\times M_{\text{cts}}(G^{n-1},A)_{\text{po}}  \arrow[]{d}[]{cb_{n,A}}\arrow{r}[]{\varphi_{\ast} \times \varphi_{\ast} } & Z^{n}_{\text{cts}}(G,A')_{\text{po}}\times M_{\text{cts}}(G^{n-1},A')_{\text{po}}  \arrow{d}[]{cb_{n,A'}} \\ Z^{n}_{\text{cts}}(G,A)_{\text{po}} \arrow{r}[]{\varphi_{\ast}} & Z^{n}_{\text{cts}}(G,A')_{\text{po}}.
		\end{tikzcd} \]
		Therefore, $ \varphi $ induces a commutative diagram
		\[ 
		\begin{tikzcd}[row sep=1.5em , column sep=2.0em]
		Z^{n}_{\text{cts}}(G,A)_{\text{po}} \arrow[two heads]{d}[]{\pi_{n,A}}\arrow{r}[]{\varphi_{\ast} } & Z^{n}_{\text{cts}}(G,A')_{\text{po}} \arrow[two heads]{d}[]{\pi_{n,A'}} \\ H^{n}_{\text{cts}}(G,A)_{\text{po}} \arrow{r}[]{\overline{\varphi}_{\ast}} & H^{n}_{\text{cts}}(G,A')_{\text{po}},
		\end{tikzcd} \]
		where for $ [a] \in H^{n}_{\text{cts}}(G,A)_{\text{po}} $ we set $ \overline{\varphi}_{\ast}([a])\stackrel{\mathrm{def}}{=} [\varphi_{\ast}(a)] $. The map $ \overline{\varphi}_{\ast} $ is continuous, which one can check with similar calculation as done in Remark \ref{R4}. 
		
		Similarly, if $ \psi : H\longrightarrow G $ is a continuous group homomorphism, one obtains a commutative diagram of continuous maps
		\[ 
		\begin{tikzcd}[row sep=1.5em , column sep=2.0em]
		Z^{n}_{\text{cts}}(G,A)_{\text{po}} \arrow[two heads]{d}[]{\pi_{n,A}}\arrow{r}[]{\psi^{\ast}} & Z^{n}_{\text{cts}}(H,A)_{\text{po}} \arrow[two heads]{d}[]{\pi_{n,A}} \\ H^{n}_{\text{cts}}(G,A)_{\text{po}} \arrow{r}[]{\overline{\psi}^{\ast}} & H^{n}_{\text{cts}}(H,A)_{\text{po}},
		\end{tikzcd} \] 
		where for $ [a] \in H^{n}_{\text{cts}}(G,A)_{\text{po}} $ we set $ \overline{\psi}^{\ast}([a])\stackrel{\mathrm{def}}{=} [\psi^{\ast}(a)] $. 	  
	\end{remark}
	\begin{lemma}\label{L14}
		Let $ A $ be a compact $ G $-module, fix $ n\geq1 $ and assume $ A $ is evenly continuous with respect to $ G^{n-1} $. Then the following facts hold:
		\begin{enumerate}[i)]
			\item The topological group $ H_{\text{cts}}^{n}(G,A)_{\text{po}} $ is Hausdorff.
			\item The topological group $ Z_{\text{cts}}^{n}(G,A)_{\text{po}} $ is compact if and only if the topological groups $ H_{\text{cts}}^{n}(G,A)_{\text{po}} $ is compact; in particular if $ A $ is evenly continuous with respect to $ G^{n-1} $, then $ H_{\text{cts}}^{n}(G,A)_{\text{po}} $ is a compact space.	
		\end{enumerate} 
	\end{lemma}
	\begin{proof}
		By Proposition \ref{P3} the cohomology map $ cb_{n} $ is continuous. Since $ A $ is compact $ G $-module, which is evenly continuous with respect to $ G^{n-1} $, by Lemma \ref{L3}, i), and Lemma \ref{L5}, the space $ M_{\text{cts}}(G^{n-1},A)_{\text{po}} $ is a compact and Hausdorff space, also $ Z_{\text{cts}}^{n}(G,A)_{\text{po}} $ is a Hausdorff group. Now $ M_{\text{cts}}(G^{n-1},A)_{\text{po}} $ is a compact and Hausdorff topological group, which acts continuously on a Hausdorff space $ Z_{\text{cts}}^{n}(G,A)_{\text{po}} $. Therefore, $ Z_{\text{cts}}^{n}(G,A)_{\text{po}} $ is a $ M_{\text{cts}}(G^{n-1},A)_{\text{po}}  $-space with $ M_{\text{cts}}(G^{n-1},A)_{\text{po}}  $ a compact space, which matches the definition of \cite[p.~32, 1]{MR0413144}. Now i) follows from \cite[p.~39, 3.1 Theorem, (1)]{MR0413144}, and ii) follows from \cite[p.~39, 3.1 Theorem, (4)]{MR0413144}; in particular, if $ A $ is evenly continuous with respect to $ G^{n-1} $, then the space $ Z_{\text{cts}}^{n}(G,A)_{\text{po}} $ is a compact space, which implies that the space $ H_{\text{cts}}^{n}(G,A)_{\text{po}} $ is compact.
	\end{proof} 
	
	Now we prove Theorem \ref{T2}. The proof is similar the proof of Theorem \ref{T1}, and it is established using 3 lemmas; by showing there exists a well-defined natural continuous bijection $ \Theta_{n} $.
	\begin{lemma}\label{L15}
		Under the assumption of Theorem \ref{T2}, there exists a well-defined continuous map \[ \Theta_{n} : H^{n}_{\text{cts}}(G,A)_{\text{po}} \longrightarrow \varprojlim_{\stackrel{ \ }{r\in R}} H^{n}_{\text{cts}}(G,A_{r})_{\text{po}}  \ \ \text{given by} \ \ [a]\longmapsto ([\varphi_{r\ast}(a)]), \] such that following diagram is commutative
		\[ 
		\begin{tikzcd}[row sep=1.5em , column sep=2.0em]
		Z_{\text{cts}}^{n}(G,A)_{\text{po}} \arrow[two heads]{d}[]{\pi_{n,A}}\arrow{r}[]{\theta_{n}} & \varprojlim_{\stackrel{ \ }{r\in R}} Z_{\text{cts}}^{n}(G,A_{r})_{\text{po}} \arrow{d}[]{\varprojlim_{\stackrel{ \ }{r\in R}}\pi_{n,A_{r}}} \\ H^{n}_{\text{cts}}(G,A)_{\text{po}} \arrow{r}[]{\Theta_{n}} & \varprojlim_{\stackrel{ \ }{r\in R}} H^{n}_{\text{cts}}(G,A_{r})_{\text{po}}.
		\end{tikzcd} \]
	\end{lemma}
	\begin{proof}
		Since $ \varphi_{rt}: A_{r}\longrightarrow A_{t} $ for $ r\geq t $ are continuous $ G $-group homomorphisms, by Remark \ref{R7} they induce a commutative diagram of continuous maps
		\[ 
		\begin{tikzcd}[row sep=1.5em , column sep=2.0em]
		Z_{\text{cts}}^{n}(G,A_{r})_{\text{po}} \arrow[two heads]{d}[]{\pi_{n,A_{r}}}\arrow{r}[]{\varphi_{rt\ast}} & Z_{\text{cts}}^{n}(G,A_{t})_{\text{po}} \arrow[two heads]{d}[]{\pi_{A_{t}}} \\ H^{n}_{\text{cts}}(G,A_{r})_{\text{po}} \arrow{r}[]{\overline{\varphi}_{rt\ast}} & H^{n}_{\text{cts}}(G,A_{t})_{\text{po}}.
		\end{tikzcd} \] 
		Note $ \varphi_{rt} $ are transition maps, and so $ \{H^{n}_{\text{cts}}(G,A_{r})_{\text{po}}, \overline{\varphi}_{rt\ast}, R \} $ is an inverse system of topological groups. 
		
		Now for $ r \in R $ the continuous maps $ \varphi_{r}: A\longrightarrow A_{r} $ are $ G $-group homomorphisms, so by Remark \ref{R7} they induces a commutative diagrams of continuous maps
		\[ 
		\begin{tikzcd}[row sep=1.5em , column sep=2.0em]
		Z_{\text{cts}}^{n}(G,A)_{\text{po}} \arrow[two heads]{d}[]{\pi_{n,A}}\arrow{r}[]{\varphi_{r\ast}} & Z_{\text{cts}}^{n}(G,A_{r})_{\text{po}} \arrow[two heads]{d}[]{\pi_{n,A_{r}}} \\ H^{n}_{\text{cts}}(G,A)_{\text{po}} \arrow{r}[]{\overline{\varphi}_{r\ast}} & H^{n}_{\text{cts}}(G,A_{r})_{\text{po}},
		\end{tikzcd} \]
		compatible with transition maps of the inverse systems $ \{Z^{n}_{\text{cts}}(G,A_{r})_{\text{po}}, \varphi_{rt\ast}, R \} $ and $ \{H^{n}_{\text{cts}}(G,A_{r})_{\text{po}}, \overline{\varphi}_{rt\ast}, R \} $. Taking the inverse limit of the above diagram we obtain a unique commutative diagram 
		\[ 
		\begin{tikzcd}[row sep=1.5em , column sep=2.0em]
		Z_{\text{cts}}^{n}(G,A)_{\text{po}} \arrow[two heads]{d}[]{\pi_{n,A}}\arrow{r}[]{\theta_{n}} & \varprojlim_{\stackrel{ \ }{r\in R}} Z_{\text{cts}}^{n}(G,A_{r})_{\text{po}} \arrow{d}[]{\varprojlim_{\stackrel{ \ }{r\in R}}\pi_{n,A_{r}}} \\ H^{n}_{\text{cts}}(G,A)_{\text{po}} \arrow{r}[]{\Theta_{n}} & \varprojlim_{\stackrel{ \ }{r\in R}} H^{n}_{\text{cts}}(G,A_{r})_{\text{po}}.
		\end{tikzcd} \] 
		This shows that the map $ \Theta_{n} $ is well-defined and is given by $ [a]\longmapsto ([\varphi_{r\ast}(a)]) $ for $ [a] \in H^{n}_{\text{cts}}(G,A)_{\text{po}} $. Now it follows from Lemma \ref{L4} that the map $ \theta_{n} $ is continuous (a homeomorphism in fact since $ A $ is compact), and $ \varprojlim_{\stackrel{ \ }{r\in R}}\pi_{n,A_{r}} $ is a continuous map (being inverse limit of continuous maps), also since $ \pi_{n,A} $ is an open map, we find that $ \Theta_{n} $ is a continuous map. 
	\end{proof} 
	\begin{lemma}\label{L16}
		Under the assumption of Theorem \ref{T2} the continuous map \[ \Theta_{n} : H^{n}_{\text{cts}}(G,A)_{\text{po}} \longrightarrow \varprojlim_{\stackrel{ \ }{r\in R}} H^{n}_{\text{cts}}(G,A_{r})_{\text{po}}, \] 
		defined in Lemma \ref{L15}, is injective.
	\end{lemma}
	\begin{proof}
		Let $ [a], [b] \in H^{n}_{\text{cts}}(G,A)_{\text{po}} $, and suppose \[ \Theta_{n}([a])=([\varphi_{r}a])=\Theta_{n}([b])=([\varphi_{r}b]). \] Then for all $ r \in R $ there exists $ f_{r} \in M_{\text{cts}}(G^{n-1},A_{r})_{\text{po}} $ such that $ \varphi_{r}a=\varphi_{r}b + d_{n}f_{r} $. Therefore, the set $ S_{n,r}=\{\widetilde{f} \in M_{\text{cts}}(G^{n-1},A_{r})_{\text{po}}: \varphi_{r}a=\varphi_{r}b+ d_{n}\widetilde{f} \} $ is nonempty for all $ r \in R $. Now $ S_{n,r} $ is equal to the coset $ \text{Stab}_{n,A_{r}}(\varphi_{r}b)+f_{r} $, where $ \text{Stab}_{n,A_{r}}(\varphi_{r}b) $ is the stabiliser of $ \varphi_{r}b $ for the action of $ M_{\text{cts}}(G^{n-1},A_{r})_{\text{po}} $ on $ Z_{\text{cts}}^{n}(G,A_{r})_{\text{po}} $. The set $ \text{Stab}_{n,A_{r}}(\varphi_{r}b) \subseteq M_{\text{cts}}(G^{n-1},A_{r})_{\text{po}} $ is a closed, compact, and Hausdorff subgroup by Corollary \ref{C6}.
		Therefore, $ S_{n,r}=\text{Stab}_{n,A_{r}}(\varphi_{r}b) + f_{r} $ is nonempty, compact, and Hausdorff for each $ r \in R $. In particular, if $ r\geq t $ and $ \widetilde{f} \in S_{n,r} $, then \[ \varphi_{t\ast}(b) + d_{n}\varphi_{rt\ast}(\widetilde{f})=\varphi_{rt\ast}\varphi_{r\ast}(b)  +d_{n}\varphi_{rt\ast}(\widetilde{f})=\varphi_{rt\ast}(\varphi_{r\ast}(b)  +d_{n}\widetilde{f})=\varphi_{rt\ast}\varphi_{r\ast}(a)=\varphi_{t\ast}(a), \]
		so $ \varphi_{rt\ast}(\widetilde{f}) \in S_{n,t} $. Hence, the set $ \{S_{n,r}, \varphi_{rt\ast}, R\} $ is an inverse system of nonempty, compact, and Hausdorff spaces. Therefore, $ \varprojlim_{\stackrel{ \ }{r \in R}} S_{r} $ is nonempty, compact, and Hausdorff cf. \cite[p.~4, Proposition 1.1.4]{MR2599132}. Now for \[ f \in \varprojlim_{\stackrel{ \ }{r \in R}} S_{r} \hookrightarrow \varprojlim_{\stackrel{ \ }{r \in R}} M_{\text{cts}}(G^{n-1},A_{r})_{\text{po}} \cong M_{\text{cts}}(G^{n-1},A)_{\text{po}},\] one has \[ \varphi_{r\ast}( b+d_{n}f)=\varphi_{r\ast}(b) \ +d_{n}\varphi_{r\ast}(f)=\varphi_{r\ast}(a) \] for all $ r \in R $, so $ \theta_{n}(b+d_{n}f)=\theta_{n}(a) $, and since it follows from Lemma \ref{L4} that the map $ \theta_{n} $ is injective, we must have $ b+d_{n}f=a $; therefore $ [a]=[b] $, and this proves that the map $ \Theta_{n} $ is injective. 
	\end{proof}
	\begin{lemma}\label{L17}
		Under the assumption of Theorem \ref{T2}. Then the continuous map \[ \Theta_{n} : H^{n}_{\text{cts}}(G,A)_{\text{po}} \longrightarrow \varprojlim_{\stackrel{ \ }{r\in R}} H^{n}_{\text{cts}}(G,A_{r})_{\text{po}}, \] 
		defined in Lemma \ref{L15}, is surjective.
	\end{lemma}
	\begin{proof}
		Recall we have surjective continuous maps \[\pi_{n,A_{r}}: Z^{n}_{\text{cts}}(G,A_{r})_{\text{po}} \longrightarrow H^{n}_{\text{cts}}(G,A_{r})_{\text{po}}  \] 
		for  $ r \in R $, which are compatible with the transition maps of $ \{Z^{n}_{\text{cts}}(G,A_{r})_{\text{po}}, \varphi_{rt\ast}, R \} $ and $ \{H^{n}_{\text{cts}}(G,A_{r})_{\text{po}}, \overline{\varphi}_{rt\ast}, R \} $. Let $ ([a_{r}]) \in \varprojlim_{\stackrel{ \ }{r\in R}} H^{n}_{\text{cts}}(G,A_{r})_{\text{po}} $. Then since $ A_{r} $ is compact and Hausdorff and evenly continuous with respect to $ G^{n-1} $ for each $ r \in R $, the set $ \pi_{n,A_{r}}^{-1}([a_{r}])=\text{Orb}_{n,A_{r}}(a_{r}) $ is a nonempty, compact, and Hausdorff space for each $ r \in R $ by Corollary \ref{C7}. In particular the set $ \{\text{Orb}_{n,A_{r}}(a_{r}), \overline{\varphi}_{rt\ast}, R\} $ forms an inverse system of nonempty, compact, and Hausdorff spaces. Therefore, $ \varprojlim_{\stackrel{ \ }{r \in R}} \text{Orb}_{n,A_{r}}(a_{r}) $ is nonempty, compact, and Hausdorff cf. \cite[p.~4, Proposition 1.1.4]{MR2599132}. Now for \[ a \in \varprojlim_{\stackrel{ \ }{r \in R}}\text{Orb}_{n,A_{r}}(a_{r}) \hookrightarrow \varprojlim_{\stackrel{ \ }{r \in R}} Z^{n}_{\text{cts}}(G,A_{r})_{\text{po}} \] we have $ \pi_{n,A}\theta^{-1}(a) \in H^{n}_{\text{cts}}(G,A)_{\text{po}} $ and $ \Theta_{n}\pi_{n,A}\theta_{n}^{-1}(a)=([a_{r}]) $; this proves that the map $ \Theta_{n} $ is surjective. 
	\end{proof}
	
	Now we provide the proof of the Theorem \ref{T2}. We recall the statement of Theorem \ref{T2} again.
	\begin{theorem}[]
		Let $ A $ a compact $ G $-module and fix $ n\geq 1 $. Assume $ A $ has a presentation $ A= \varprojlim_{\stackrel{ \ }{r\in R}}A_{r} $, where $ A_{r} $ is evenly continuous with respect to $ G^{n-1} $ for all $ r\in R $. Denote by $ \varphi_{r}: A\longrightarrow A_{r} $ the natural projections. Then there exists a continuous bijection \[ \Theta_{n} : H_{\text{cts}}^{n}(G,A)_{\text{po}} \longrightarrow \varprojlim_{\stackrel{ \ }{r\in R}} H_{\text{cts}}^{n}(G,A_{r})_{\text{po}} \ \ \text{given by} \ \ [a]\longmapsto ([\varphi_{r\ast}(a)]). \]
		In particular, if $ A $ is evenly continuous with respect to $ G^{n} $, then $ \Theta_{n} $ is a homeomoprhism.
	\end{theorem}
	\begin{proof}
		By Lemma \ref{L15} there exists a commutative diagram of continuous maps
		\[ 
		\begin{tikzcd}[row sep=1.5em , column sep=2.0em]
		Z^{n}(G,A)_{\text{co}} \arrow[two heads]{d}[]{\pi_{n,A}}\arrow{r}[]{\theta_{n}} & \varprojlim_{\stackrel{ \ }{r\in R}} Z^{n}(G,A_{r})_{\text{co}} \arrow{d}[]{\varprojlim_{\stackrel{ \ }{r\in R}}\pi_{n,A_{r}}} \\ H^{n}_{cts}(G,A) \arrow{r}[]{\Theta_{n}} & \varprojlim_{\stackrel{ \ }{r\in R}} H^{n}_{cts}(G,A_{r}),
		\end{tikzcd} \]
		by Lemma \ref{L16} the map $ \Theta_{n} $ is injective, and by Lemma \ref{L17} the map $ \Theta_{n} $ is surjective. Therefore, $ \Theta_{n} $ is a continuous bijection. Using Lemma \ref{L14}, i), we deduce that $ \varprojlim_{\stackrel{ \ }{r\in R}} H^{n}_{\text{cts}}(G,A_{r}) $ is a Hausdorff space for each $ r \in R $. In particular, if $ A $ is evenly continuous with respect to $ G^{n} $, then using Lemma \ref{L5} and Lemma \ref{L14}, ii), we have that $ H^{n}_{\text{cts}}(G,A) $ is a compact space; therefore, $ \Theta_{n} $ is a continuous bijection from a compact space to a Hausdorff space, which implies that  $ \Theta_{n} $ is a homeomorphism cf. \cite[p.~123, 17.14 Theorem]{MR2048350}. 
	\end{proof}
	\subsection{One Application of The Second Main Theorem}\label{SB7} 
	One immediate application of the Theorem \ref{T2} is the following. Assume $ G $ is a profinite group.
	\begin{corollary}\label{C8}
		Let $ A $ be a profinite $ G $-module. Then $ A $ is an inverse limit of finite discrete $ G $-modules i.e., $ A $ has a presentation $ A\cong\varprojlim_{\stackrel{ \ }{r\in R}}A_{r} $ satisfying the assumptions of Theorem \ref{T2}, and if $ A_{r} $ is evenly continuous with respect to $ G^{n-1} $, then there exists a continuous bijection \[ \Theta_{n} : H_{\text{cts}}^{n}(G,A)_{\text{po}} \longrightarrow \varprojlim_{\stackrel{ \ }{r\in R}} H_{\text{cts}}^{n}(G,A_{r})_{\text{po}} \ \ \text{given by} \ \ [a]\longmapsto ([\varphi_{r,\ast}(a)]); \]
		which is homeomorphism if $ A $ is evenly continuous with respect to $ G^{n} $.    
	\end{corollary}
	\begin{proof}
		It follows from \cite[p.~170, Lemma 5.3.3, (c)]{MR2599132} that $ A $ is an inverse limit of finite discrete $ G $-modules; therefore has a presentation $ A\cong \varprojlim_{\stackrel{ \ }{r\in R}}A_{r} $ by finite discrete $ G $-modules; and the corollary can be deduced from Theorem \ref{T2}. 
	\end{proof}
	
	Recall the theorem \cite[p.~141, (2.7.5) Theorem]{MR2392026}, which states that if $ G $ is a profinite group, and $ A $ is a compact $ G $-module having a presentation $ A=\varprojlim_{n \in \mathbb{N}}A_{n} $ as a countable inverse limit of finite discrete $ G $-modules $ A_{n} $, then there exist a natural exact sequence  \[ 
	\begin{tikzcd}[row sep=1.2em , column sep=1.2em]
	1 \arrow{r}[]{}\arrow{r}{}[swap]{} & \varprojlim_{n \in \mathbb{N}}^{1}H^{i-1}(G,A_{n})   \arrow{r}[]{} & H^{i}_{\text{cts}}(G,A) \arrow{r}[]{} & \varprojlim_{n \in \mathbb{N}}H^{i}(G,A_{n}) \arrow{r}[]{} & 1. 
	\end{tikzcd} \] 
	Therefore, Corollary \ref{C7} shows that if $ A_{n} $ is evenly continuous with respect to $ G^{i-1} $ for all $ n \in \mathbb{N} $, then $ \varprojlim_{n \in \mathbb{N}}^{1}H^{i-1}(G,A_{n}) =1 $.
\bibliographystyle{alpha} 
\bibliography{KNZccepot}
\end{document}